\documentclass[9pt,twocolumn]{IEEEtran}

\usepackage{amsmath,graphicx,amsfonts,amssymb,epsfig,mathrsfs}
 \usepackage{amsthm}
\usepackage{subcaption}
\usepackage{color}
\usepackage{multirow,multicol,balance}
\usepackage{rotating}
\usepackage{algorithm}
\usepackage{algpseudocode}
\usepackage{algorithmicx}
\usepackage{cite}
\usepackage{url}
\usepackage{framed}
\usepackage{bm}

\newtheorem{theorem}{Theorem}

\newtheorem{definition}{Definition}

\newtheorem{proposition}{Proposition}
\newtheorem{remark}{Remark}

\newtheorem{example}{Example}

\usepackage{tikz}
\usetikzlibrary{calc,trees,positioning,arrows,chains,shapes.geometric,decorations.pathreplacing,decorations.pathmorphing,shapes, matrix,shapes.symbols}

\newcommand{\bq}{{\bm{q}}}

\newcommand{\be}{{\bm{e}}}

\newcommand{\bx}{{\bm{x}}}

\newcommand{\by}{{\bm{y}}}
\newcommand{\bz}{{\bm{z}}}
\newcommand{\bbf}{{\bm{f}}}

\newcommand{\differential}{{\rm{d}}}

\newcommand{\bA}{\bm{A}}
\newcommand{\bM}{\bm{M}}
\newcommand{\bb}{\bm{b}}

\renewcommand{\det}{{\mathrm{det}}}

\def\spacingset#1{\def\baselinestretch{#1}\small\normalsize}
\spacingset{1}

\allowdisplaybreaks

\title{\huge{
Finite Horizon Density Control for Static State Feedback\\Linearizable Systems}
}

\author{Kenneth F. Caluya, and Abhishek Halder
\thanks{Kenneth F. Caluya, and Abhishek Halder are with the Department of Applied Mathematics, University of California, Santa Cruz, CA 95064, USA,
        {\tt\small{\{kcaluya,ahalder\}@ucsc.edu}}. This research was partially supported by NSF award 1923278.%
}}

\begin{document}

\markboth{\today}{}

\maketitle

\begin{abstract}
We consider the problem of steering the joint state probability density function of a static feedback linearizable control system over finite time horizon. Potential applications include controlling neuronal populations, swarm guidance, and probabilistic motion planning. Our theoretical developments reveal the structure of the minimum energy controller for the same, and can be viewed as a generalization of the Benamou-Brenier theory for dynamic optimal transport. Further analytical results are derived for solving the feasibility problem, i.e., for finding  feedback that steers a given joint density function to another in fixed time, subject to the controlled nonlinear dynamics. An  algorithm based on the Schr\"{o}dinger bridge is proposed to approximate a feasible controller; a numerical example is worked out to illustrate the same. 
\end{abstract}




\section{Introduction}
Steering the joint distribution of the state vector $\bx(t)$ of a controlled dynamical system from a prescribed distribution to another over a finite time horizon (say, $t\in[0,1]$) is an emerging research area \cite{brockett2007optimal,mazurenko2011dynamic,brockett2012notes}, with applications in controlling robotic swarms \cite{pimenta2008control}, shaping the bulk magnetization distribution for Nuclear Magnetic Resonance (NMR) spectroscopy \cite{li2009ensemble}, controlling neuronal populations \cite{monga2018synchronizing}, and process control \cite{wang2004reachability}. These applications concern population distribution whose shape is actively controlled over time while preserving the physical mass (say, of unit amount, without loss of generality). The conservation of mass allows an alternative interpretation of the underlying mathematical problem -- instead of steering a large number of systems with identical dynamics, one can think of steering a single system with probabilistic uncertainty in its initial and terminal state, modeled via prescribed initial and terminal joint state probability density functions (PDFs), viz. $\rho_{0}(\bx)$ at $t=0$, and $\rho_{1}(\bx)$ at $t=1$. This latter interpretation corresponds to atypical stochastic control problem since unlike the classical two-point boundary value problems in finite dimensional vector space, now the boundary ``values" are measure-valued\footnote{Hereafter, we will tacitly assume that the underlying probability measures are absolutely continuous, i.e., the joint PDFs exist.}. Therefore, we are led to solving two point boundary value problem on an infinite dimensional manifold. In the robotics literature, control in the space of joint PDFs is often referred to as the ``belief space" control problem; see e.g., \cite{prentice2009belief}. The purpose of this paper is to design control input for feedback linearizable systems to steer $\rho_{0}(\bx)$ to $\rho_{1}(\bx)$.

In the systems-control literature, two broad design approaches have appeared for density control. In the so-called ``ensemble control" \cite{li2011ensemble,becker2012approximate} approach, one designs \emph{open-loop control} $\bm{u}(t)$ in the sense that at any given time $t$, the same control is applied at all state space locations, i.e., $\bm{u}(t)$ is a broadcast. This paper follows the other approach, where one designs (possibly mixed feed-forward) \emph{feedback control} $\bm{u}(\bm{x},t)$, i.e., the control is spatially inhomogeneous. 

Designing feedback for finite horizon PDF shaping is closely related to the classical optimal transport problem \cite{villani2003topics}. Specifically, the dynamic formulation \cite{benamou2000computational} of the optimal transport, in controls language, is a problem of determining the minimum energy input $\bm{u}(\bm{x},t)$ that steers the joint PDF $\rho_{0}(\bx)$ to $\rho_{1}(\bx)$ over $t\in[0,1]$ subject to \emph{zero prior dynamics}, i.e., it solves\footnote{Throughout this paper, the notation $\bm{x}\sim\rho$ means that the random vector $\bx$ has joint PDF $\rho$.} 
\begin{subequations}
\begin{align}	
&\qquad\quad\underset{\bm{u}}{\text{inf}}
\quad\mathbb{E} \bigg \{  \int_{0}^{1}  \lVert \bm{u}(\bx,t) \rVert_{2}^{2} \: \differential t \bigg \} \label{OMTobj}\\ 
&\text{subject to} \quad  \dot{\bx} = \bm{u},\label{OMTconstrDyn}\\
& \qquad\qquad\quad \bx(0) \sim \rho_{0}(\bx), \quad \bx(1) \sim \rho_{1}(\bx),\label{OMTconstrBdy}
\end{align} 
\label{OMT}	
\end{subequations}
where the endpoint joint PDFs $\rho_{0},\rho_{1}$ are assumed to have finite second moments. In (\ref{OMTobj}), the expectation operator $\mathbb{E}\{\cdot\}$ is taken with respect to (w.r.t.) the controlled state PDF $\rho(\bx,t)$ satisfying $\rho(\bx,0)=\rho_{0}(\bx)$ and $\rho(\bx,1)=\rho_{1}(\bx)$. For (\ref{OMT}), the existence and uniqueness of the optimal control $\bm{u}^{\text{opt}}(\bx,t)$ are guaranteed, and $\bm{u}^{\text{opt}}$ is known to be a (non-autonomous) gradient vector field \cite{benamou2000computational}. This result has been generalized \cite{chen2017optimal} for the case when (\ref{OMTconstrDyn}) is replaced by a \emph{controlled linear system} $\dot{\bx}(t) = \bm{A}(t)\bx + \bm{B}(t)\bm{u}$, where $\bx \in \mathbb{R}^{n}$, $\bm{u}\in\mathbb{R}^{m}$, and the pair $(\bm{A}(t),\bm{B}(t))$ is controllable, and also for the case \cite{chen2018optimal} when the ``cost-to-go" in (\ref{OMTobj}) involves an additional term that is quadratic in state $\bx$. A related work \cite{halder2016finite} derived the optimal controller by penalizing a PDF-level terminal cost instead of enforcing a terminal PDF constraint. PDF tracking controllers \cite{halder2014geodesic,chen2018state} have also appeared for the controlled linear dynamics. Also noteworthy are the works on finite horizon covariance control \cite{grigoriadis1997minimum,bakolas2018finite,okamoto2018optimal}.

Much less studied is the case when the prior dynamics is \emph{nonlinear}; relevant works include \cite{agrachev2009optimal} and \cite{elamvazhuthi2019optimal}. A major impediment for the nonlinear case is the issue of reachability, i.e., whether $\rho_{1}$ is reachable from $\rho_{0}$ in unit time. Even if reachability can be guaranteed, it is less obvious how to generalize the development in \cite{chen2017optimal}. Our intent in this paper is to address these issues for feedback linearizable systems. For this class of nonlinear control systems, we show that one can steer $\rho_{0}$ to $\rho_{1}$ in finite time, and develop the theory for the minimum energy controller (Section \ref{SectionMinEnergySteering}). If one is willing to dispense the minimum energy criterion, then we show a strategy to derive the explicit form of the controller (Section \ref{SecSuboptimal}) solving the feasibility problem, i.e., steering $\rho_{0}$ to $\rho_{1}$ in finite time subject to the prescribed feedback linearizable dynamics. Furthermore, using a stochastic dynamical regularization, we provide a computational framework (Section \ref{SecSchrodinger}) for the controlled joint state PDF evolution via the Schr\"{o}dinger bridge \cite{schrodinger1931umkehrung,schrodinger1932theorie}. A numerical example is worked out (Section \ref{SecNumExample}) to illustrate the ideas proposed herein. Thus, the paper is structured as follows. In Section II, we summarize the requisite background on state feedback linearizable systems for the single input case. Section III provides the stochastic optimal control formulation, its reformulation via feedback linearizing transformation, the necessary conditions of optimality, as well as representation formulas for the associated value function. In Section IV, we show that how ideas from the theory of optimal mass transport and feedback linearization can be brought together to derive a feasible controller. Section V shows that a dynamic stochastic regularization can help in computing the aforesaid feasible controller. Section VI provides an illustrative numerical example while Section VII concludes the paper.

Before delving into the details, let us comment on the practical scope of the problem considered herein. With the growing interest in controlling a large population of autonomous ground and aerial vehicles, a natural question is whether one could exploit the structural aspects of their trajectory-level dynamical nonlinearities in density control. Many of these systems are known to be differentially flat \cite{fliess1995flatness, murray1995differential, fliess1999lie} -- an aspect that has been utilized in the robotics-control literature for efficient motion planning \cite{ross2002pseudospectral,mellinger2011minimum,hwan2013optimal}. Every feedback linearizable system is differentially flat; conversely, it is known \cite[Theorem 4.1]{van1998differential} that every differentially flat system can be put in Brunovsky normal (i.e., chain of integrator) form in an open and dense set through regular endogenous (possibly dynamic) feedback. In particular, also known is the fact that single input differential flatness is equivalent to being static state feedback linearizable \cite[Theorem 5.3]{van1998differential}. Thus, the developments in this paper are expected to lay the foundation for belief space motion planning for many autonomous systems of practical interest.

\subsubsection*{\!\!\!\!\!\!\!\!\!\!\!\!\!\!Notations and nomenclature} We use $\bm{0}$ to denote the column vector of appropriate dimension containing all zeros. The symbols $\bm{e}_{1}, \hdots, \bm{e}_{n}$ denote the (column) basis vectors in $\mathbb{R}^{n}$. We use $\circ$, ${\rm{spt}}(\cdot)$, and $\sharp$ to respectively denote the function composition, support of a function, and push-forward of a PDF. The standard Euclidean inner product, gradient, Laplacian, Hessian, and determinant operators are denoted by $\langle\cdot,\cdot\rangle$, $\nabla$, $\Delta$, ${\rm{Hess}}(\cdot)$, and $\det(\cdot)$, respectively. For $n\times n$ symmetric positive definite matrix $\bm{\Gamma}$, we denote the weighted Euclidean inner product between $\bx,\by\in\mathbb{R}^{n}$ as $\langle \bx, \by\rangle_{\bm{\Gamma}}:=\by^{\top}\bm{\Gamma}\bx$, and the associated weighted Euclidean norm as $\parallel\bx\parallel_{\bm{\Gamma}}:=\sqrt{\langle \bx, \bx\rangle_{\bm{\Gamma}}}$. Of course, the standard inner product $\langle\cdot,\cdot\rangle \equiv \langle\cdot,\cdot\rangle_{\bm{I}}$, i.e., when $\bm{\Gamma}\equiv\bm{I}$ (the identity matrix). We sometimes put a subscript to the expectation operator $\mathbb{E}\{\cdot\}$ to clarify w.r.t. which probability measure the expectation is taken. We drop the subscript when there is no potential confusion.

The Lie bracket of two vector fields $\bm{\xi}$ and $\bm{\eta}$ at $\bx\in\mathbb{R}^{n}$ is a new vector field $\left[\bm{\xi},\bm{\eta}\right](\bx):=(\nabla_{\bx}\bm{\eta})\bm{\xi}(\bx) - (\nabla_{\bx}\bm{\xi})\bm{\eta}(\bx)$. For $k\in\mathbb{N}$, the $k$-fold Lie bracketing of $\bm{\eta}$ with the same vector field $\bm{\xi}$ is denoted as ${\rm{ad}}_{\bm{\xi}}^{k}\bm{\eta}:=\left[\bm{\xi},{\rm{ad}}_{\bm{\xi}}^{k-1}\bm{\eta}\right]$; by convention ${\rm{ad}}_{\bm{\xi}}^{0}\bm{\eta}=\bm{\eta}$. The Lie derivative of a scalar-valued function $\lambda(\bx)$ w.r.t. the vector field $\bm{\xi}$ evaluated at $\bx$ is $L_{\bm{\xi}}\lambda(\bx):=\langle\nabla_{\bx}\lambda,\bm{\xi}\rangle(\bx)$. For $k\in\mathbb{N}$, the $k$-fold Lie derivative of $\lambda$ w.r.t. the same vector field $\bm{\xi}$ evaluated at $\bx$ is denoted as $L_{\bm{\xi}}^{k}\lambda(\bx):=\langle\nabla_{\bx}L_{\bm{\xi}}^{k-1}\lambda,\bm{\xi}\rangle(\bx)$; by convention $L_{\bm{\xi}}^{0}\lambda(\bx)=\lambda(\bx)$. 

The Legendre-Fenchel conjugate of a function $\omega : \mathbb{R}^{n} \mapsto \mathbb{R}$, is a convex function $\omega^{*} : \mathbb{R}^{n} \mapsto \mathbb{R}$, given by
\begin{align*}
\omega^{*}(\bq) := \underset{\bm{p}\in\mathbb{R}^{n}}{\sup}\{\langle\bm{q},\bm{p}\rangle - \omega(\bm{p})\}.	
\end{align*}
For a smooth manifold $\mathcal{M}$, the symbol $\mathcal{T}_{\bx}\mathcal{M}$ denotes the tangent space at $\bx\in\mathcal{M}$, and $\mathcal{T}\mathcal{M}:=\{(\bx,\bm{v})\mid\bx\in\mathcal{M},\bm{v}\in\mathcal{T}_{\bx}\mathcal{M}\}$ denotes the tangent bundle. We use the symbol $\mathcal{T}^{*}\mathcal{M}$ to denote the cotangent bundle. We use $\mathcal{N}\left(\bm{\mu},\bm{\Sigma}\right)$ to denote a multivariate Gaussian PDF with mean $\bm{\mu}$ and covariance $\bm{\Sigma}$. Standard abbreviations PDE, ODE and SDE stand for partial, ordinary and stochastic differential equation, respectively.


\section{Static Feedback Linearizable Systems}\label{SectionFLS}
For clarity of exposition, we consider \emph{single input} control affine systems of the form 
\begin{align}
	\dot{\bx} = \bm{f}(\bx) + \bm{g}(\bx)u, \quad \bx\in\mathbb{R}^{n},\;u\in\mathbb{R}.
	\label{SingleInputControlAffine}
\end{align}
It should be apparent from the development below that our ideas generalize for the multi-input case, but we will not pursue the generalizations here. We suppose that full state feedback is available. Next, we define the full state static feedback linearizable system, and collect some known results which will be useful in the sequel.
\begin{definition}\label{DefFeedbackLin}\cite[Ch. 4.2]{isidori1989nonlinear}
Given a point $\overline{\bx}\in\mathbb{R}^{n}$, and a neighborhood $\mathcal{X}$ of $\overline{\bx}$, the system (\ref{SingleInputControlAffine}) is said to be \emph{full state static feedback linearizable} around $\overline{\bx}$ if there exists a diffeomorphism $\bm{\tau}(\cdot)$ defined on $\mathcal{X}$, and a feedback $u = \alpha(\bx) + \beta(\bx)v$ also defined on $\mathcal{X}$, such that the corresponding closed loop system $\dot{\bx}=\bm{f}(\bx)+\bm{g}(\bx)(\alpha(\bx) + \beta(\bx)v)$, in the coordinates $\bz := \bm{\tau}(\bx)$, is linear and controllable, i.e.,
\begin{subequations}
\begin{align}
\left[\nabla_{\bx}\bm{\tau}\left(\bm{f}(\bx)+\bm{g}(\bx)\alpha(\bx)\right)\right]_{\bx = \bm{\tau}^{-1}(\bm{z})} &= \bm{Az},\\
\left[\nabla_{\bx}\bm{\tau}\left(\bm{g}(\bx)\beta(\bx)\right)\right]_{\bx = \bm{\tau}^{-1}(\bm{z})} &= \bm{b},	
\end{align}	
\label{AzbForm}
\end{subequations} 
where the matrix-vector pair $(\bm{A},\bm{b})$ satisfies the Kalman rank condition: ${\rm{rank}}\left(\bm{b} \vert \bm{Ab} \vert \hdots \vert \bm{A}^{n-1}\bm{b}\right)=n$. 
\end{definition}
\begin{remark}\label{RelDegreeRemark}
It is well-known (see e.g., \cite[Lemma 2.4]{isidori1989nonlinear}) that (\ref{SingleInputControlAffine}) is full state static feedback linearizable iff there exists scalar-valued function $\lambda(\bx)$ defined on $\mathcal{X}$ such that the input-output system
\begin{align}
\dot{\bx} = \bm{f}(\bx) + \bm{g}(\bx)u, \quad y = \lambda(\bx),
\label{IOfictitous}	
\end{align}
has relative degree\footnote{Relative degree is the number of times one needs to differentiate the (in our context, fictitious) output $y$ w.r.t. $t$ so that the input $u$ appears explicitly.} equal to $n$ at $\overline{\bx}$. 	
\end{remark}
When it exists, $\lambda(\bx)$ satisfies the following $n-1$ first order PDEs, and one PDE not equal to zero condition:
\begin{subequations}
\begin{align}
L_{\bm{g}}\lambda(\bx) = L_{{\rm{ad}}_{\bm{f}}\bm{g}}\lambda(\bx) = \hdots = L_{{\rm{ad}}_{\bm{f}}^{n-2}\bm{g}}\lambda(\bx) = 0, \label{SysPDE}\\
L_{{\rm{ad}}_{\bm{f}}^{n-1}\bm{g}}\lambda(\bx)\bigg\vert_{\bx = \overline{\bx}} \neq 0. \label{neqZero}	
\end{align}
\label{lambdacond}	
\end{subequations}
 The following constructive result allows to verify the existence of $\lambda(\bx)$ in Remark \ref{RelDegreeRemark}.
 \begin{proposition}\cite[Theorem 2.6]{isidori1989nonlinear}\label{ExistenceOflambda}
 A scalar-valued function $\lambda(\bx)$ defined on $\mathcal{X}$ satisfying (\ref{lambdacond}) exists iff the following two conditions are satisfied:
 \begin{enumerate}
 \item[(i)] ${\rm{rank}}\left(\left[\bm{g}(\bx)\:\bigg\vert\:{\rm{ad}}_{\bm{f}}\bm{g}(\bx)\:\bigg\vert\hdots\bigg\vert{\rm{ad}}_{\bm{f}}^{n-1}\bm{g}(\bx)\right]_{\bx=\overline{\bx}}\right)=n$,
 \vspace*{0.03in} 
 \item[(ii)] $\mathcal{D}(\bx):={\rm{span}}\{\bm{g},{\rm{ad}}_{\bm{f}}\bm{g},\hdots,{\rm{ad}}_{\bm{f}}^{n-2}\bm{g}\}(\bx)$ is involutive\footnote{Given $m$ vectors fields $\bm{\xi}_{1}(\bx), \hdots, \bm{\xi}_{m}(\bx)$ in $\mathbb{R}^{n}$, we say $\mathcal{D}(\bx):={\rm{span}}\{\bm{\xi}_{1}, \hdots, \bm{\xi}_{m}\}(\bx)$ is involutive at $\bx\in\mathbb{R}^{n}$, if for all $\bm{\xi}_{i}(\bx),\bm{\xi}_{j}(\bx)\in\mathcal{D}(\bx)$, we get that the Lie bracket $\left[\bm{\xi}_{i},\bm{\xi}_{j}\right](\bx)\in\mathcal{D}(\bx)$, where $i,j=1,\hdots,m$.} near $\bx=\overline{\bx}$. 	
 \end{enumerate}
 \end{proposition}
Given a system of the form (\ref{SingleInputControlAffine}), one first checks the conditions in Proposition \ref{ExistenceOflambda} to ascertain if the system is full state static feedback linearizable or not. If it is, then one solves (\ref{lambdacond}) to determine an admissible $\lambda(\bx)$ (in general, (\ref{lambdacond}) may admit non-unique solutions for $\lambda$). The feedback linearizing tuple $(\bm{\tau},\alpha,\beta)$ in (\ref{DefFeedbackLin}) can then be obtained as 
\begin{subequations}
\begin{align}
\bm{\tau}(\bx) &= \left(\lambda(\bx),L_{\bm{f}}\lambda(\bx),\hdots,L_{\bm{f}}^{n-1}\lambda(\bx)\right)^{\top},\label{ComputeTau}\\
\alpha(\bm{x}) &= -L_{\bm{f}}^{n}\lambda(\bx)/L_{\bm{g}}L_{\bm{f}}^{n-1}\lambda(\bx),\label{ComputeAlpha}\\
\beta(\bx) &= 1/L_{\bm{g}}L_{\bm{f}}^{n-1}\lambda(\bx).	\label{ComputeBeta}
\end{align}	
\label{alphabetatau}
\end{subequations}
The tuple $(\bm{\tau},\alpha,\beta)$ transforms (\ref{SingleInputControlAffine}) in state-control pair $(\bx,u)$ to the feedback linearized (Brunovsky normal) form
\begin{align}
\dot{\bm{z}} = \bm{A}\bm{z} + \bm{b}v,\;\bm{A}:=\left[\bm{0}\vert \bm{e}_{1} \vert \bm{e}_{2}\vert\hdots\vert\bm{e}_{n-1}\right],\;\bm{b}:=\bm{e}_{n},
\label{feedbacklinearizedform}	
\end{align}
in state-control pair $(\bm{z},v)$. 

\begin{remark}\label{u2vandv2u}
Since the relative degree is equal to $n$ (see Remark \ref{RelDegreeRemark}), hence $L_{\bm{g}}L_{\bm{f}}^{n-1}\lambda(\bx)\neq 0$ in $\mathcal{X}$. Therefore, (\ref{ComputeAlpha})-(\ref{ComputeBeta}) are well-defined. Furthermore, notice from Definition \ref{DefFeedbackLin} that $u=\alpha(\bx)+\beta(\bx)v$ implies $v = \gamma(\bx) + \delta(\bx)u$, where $\gamma(\bx):=-\alpha(\bx)/\beta(\bx)$ and $\delta(\bx):=1/\beta(\bx)$. Since $\beta(\bx)\neq 0$ in $\mathcal{X}$, the quantities $\gamma(\bx),\delta(\bx)$ are also well-defined therein.
\end{remark}


\section{Minimum Energy PDF Steering}\label{SectionMinEnergySteering}
In this Section, we formulate the stochastic optimal control problem, derive the conditions of optimality, and provide representation formulas for the solution of the same.
\subsection{Stochastic Control Problem}\label{SubsecStochasticControlProblem}
Motivated by \cite{chen2017optimal}, we consider the following minimum energy stochastic optimal control problem: 
\begin{subequations}
\begin{align}	
&\qquad\quad\underset{u \in \mathcal{U}}{\text{inf}}
& & \mathbb{E} \bigg \{  \int_{0}^{1} \frac{1}{2} \lvert u(\bx,t) \rvert^{2} \: \differential t \bigg \} \label{FLSobj}\\ 
& \;\;\text{subject to} & &  \dot{\bx} = \bm{f}(\bx) + \bm{g}(\bx)u,\label{FLSdyn}\\
& & &  \bx(0) \sim \rho_{0}(\bx), \quad \bx(1) \sim \rho_{1}(\bx),\label{PDFconstr2pbvp}
\end{align} 
\label{FLSprobMinEnergy}	
\end{subequations}
where $\bx\in\mathcal{X}\subseteq\mathbb{R}^{n}$, $u\in\mathbb{R}$, and (\ref{FLSdyn}) is a full state static feedback linearizable system, i.e., the vector fields $\bm{f}$ and $\bm{g}$ are sufficiently smooth satisfying the two conditions in Proposition \ref{ExistenceOflambda}. In particular, we assume $\bm{f},\bm{g}$ to be componentwise in class $C^{n+1}\left(\mathcal{X}\right)$ which ensures that the Lie derivatives in (\ref{alphabetatau}) are well-defined, and that $\alpha,\beta$ are continuous. Furthermore, we assume that $\rho_{0},\rho_{1}$ have finite second moments. The set of admissible controls $\mathcal{U}$ consists of all finite energy inputs $u:\mathcal{X}\times [0,1] \mapsto \mathbb{R}$. The control objective is to steer the joint state PDF $\rho(\bx,t)$ from the prescribed initial PDF $\rho_{0}$ at $t=0$ to the prescribed terminal PDF $\rho_{1}$ at $t=1$ while minimizing the control effort averaged over the ensemble of controlled state trajectories. We clarify here that the difference between (\ref{FLSprobMinEnergy}) and the setup considered in \cite{chen2017optimal} is that the controlled dynamics in (\ref{FLSdyn}) is nonlinear.

Problem (\ref{FLSprobMinEnergy}) can be transcribed into a ``fluid dynamics'' version \cite{benamou2000computational} given by: 
\begin{subequations}
\begin{align}
& \qquad\;\underset{(\rho,u)}{\text{inf}}
& & \int_{\mathcal{X}} \int_{0}^{1} \frac{1} {2}  \lvert u(\bx,t) \rvert^{2} \;\rho(\bx,t) \: \differential t \: \differential \bx\label{OrigOCFluids}\\ 
& \text{subject to} & &  \dfrac{\partial \rho}{\partial t} + \nabla_{\bx} \cdot \left(\left(\bm{f}(\bx) + \bm{g}(\bx)u\right) \rho\right) = 0, \label{FLSLiouville}\\
& & &  \rho(\bx,0)= \rho_{0}(\bx),\quad\rho(\bx,1) = \rho_{1}(\bx),\label{FLSpdftpbvp}
\end{align} 
\label{FLSdensityprob}
\end{subequations}
where the infimum is taken over all joint state PDF and admissible control pairs $(\rho,u)\equiv(\rho(\bx,t),u(\bx,t))$ satisfying\footnote{\label{FootnoteWeakSoln}Here we clarify that $\rho(\bx,t)$ is a weak solution of (\ref{FLSLiouville}) on the time interval $[0,1]$ if for all compactly supported test functions $\chi\in C_{c}^{\infty}\left(\mathbb{R}^{n}\times[0,1]\right)$, we have $\int_{\mathbb{R}^{n}}\chi(\bx,t=1)\rho_{1}(\bx)\differential\bx = \int_{\mathbb{R}^{n}}\chi(\bx,t=0)\rho_{0}(\bx)\differential\bx + \int_{0}^{1}\int_{\mathbb{R}^{n}}\left(\frac{\partial\chi}{\partial t} + \langle\bm{f} + \bm{g} u,\nabla_{\bx}\chi\rangle\right)\rho(\bx,t)\differential\bx\:\differential t$.} the feasibility conditions (\ref{FLSLiouville})-(\ref{FLSpdftpbvp}). Here, (\ref{FLSLiouville}) is the Liouville PDE governing the mass-preserving flow of the PDF $\rho(\bx,t)$ in $\mathcal{X}$, associated with the state ODE (\ref{FLSdyn}). Problem (\ref{FLSprobMinEnergy}) or its equivalent (\ref{FLSdensityprob}) can be interpreted as optimal transport over a static feedback linearizable system.


\subsection{Reformulation in Feedback Linearized Coordinates}\label{SubsecReformulation}

From Section \ref{SectionFLS}, recall that the map $\bz = \bm{\tau}(\bx)$ is a diffeomorphism on $\mathcal{X}\subseteq\mathbb{R}^{n}$ (the region where the full state static feedback linearization is valid). Formally, 
\begin{align}
\mathcal{X}:=\{\bx\in\mathbb{R}^{n}\mid \det(\nabla_{\bx}\bm{\tau})\neq 0\},
\label{DefineSetX}	
\end{align}
and $\bm{\tau}:\mathcal{X}\mapsto\mathcal{Z}$, where 
\begin{align}
\mathcal{Z}:=\{\bz\in\mathbb{R}^{n}\mid \bz = \bm{\tau}(\bx), \bx\in\mathcal{X}\}.
\label{DefineSetZ}	
\end{align}
Our idea now is to reformulate (\ref{FLSdensityprob}) in $\mathcal{Z}$ by pushing forward the PDFs $\rho_{0},\rho_{1}$ via $\bm{\tau}$. To this end, we make the following assumption:\begin{enumerate}
\item[\textbf{(A1)}] ${\rm{spt}}\left(\rho_{0}\right), {\rm{spt}}\left(\rho_{1}\right) \subseteq \mathcal{X}$.	
\end{enumerate}
In words, \textbf{(A1)} asserts that the supports of the prescribed initial and terminal PDFs are bounded (but may be open sets) in $\mathcal{X}$. For $i=0,1$, let $\sigma_{i} := \bm{\tau}\sharp\rho_{i}$ be the push-forward of $\rho_{i}$ under $\bm{\tau}$, meaning 
\begin{align}
\sigma_{i}(\bz) = \displaystyle\frac{\rho_{i}\left(\bm{\tau}^{-1}(\bz)\right)}{\big\lvert\det\left(\nabla_{\bx}\bm{\tau}\right)_{\bx = \bm{\tau}^{-1}(\bz)}\big\rvert}.
\label{rho2sigma}	
\end{align}
Since $\bm{\tau}$ is a diffeomorphism, \textbf{(A1)} implies that ${\rm{spt}}\left(\sigma_{i}\right) \subseteq \mathcal{Z}$ for $i=0,1$. In essence, \textbf{(A1)} ensures that (\ref{rho2sigma}) is well-defined. Notice from (\ref{ComputeTau}) and (\ref{DefineSetX}) that since $\bm{\tau}$, and hence $\mathcal{X}$ is determined by the vector fields $\bm{f},\bm{g}$, therefore, \textbf{(A1)} serves as a condition of compatibility for data $\bm{f},\bm{g},\rho_{0},\rho_{1}$ associated with problem (\ref{FLSdensityprob}).  

Letting 
\begin{align}
\alpha_{\bm{\tau}}(\cdot) := \alpha\circ\bm{\tau}^{-1}(\cdot), \quad \beta_{\bm{\tau}}(\cdot) := \beta\circ\bm{\tau}^{-1}(\cdot),
\label{ab}	
\end{align}
and then using $u(\bz) = \alpha_{\bm{\tau}}(\bz) + \beta_{\bm{\tau}}(\bz)v$ and (\ref{rho2sigma}), we rewrite (\ref{FLSdensityprob}) as 
\begin{subequations}
\begin{align}
&\qquad\;\underset{(\sigma,v)}{\text{inf}}
& & \int_{\mathcal{Z}} \int_{0}^{1} \frac{1} {2}   \mathcal{L}(\bz,v) \sigma(\bz,t) \: \differential t \: \differential\bz\label{ObjFluidz}\\ 
& \text{subject to} & &  \dfrac{\partial \sigma}{\partial t} + \nabla_{\bz} \cdot ((\bA\bz + \bb v )\sigma) = 0,\label{ConstrLiouvillez}\\
& & &  \sigma(\bz,0)=\sigma_{0}(\bz),  \quad \sigma(\bz,1)=\sigma_{1}(\bz),\label{ConstrPDFz}
\end{align}
\label{TransformedLOMT} 
\end{subequations}
where
\begin{align}
\mathcal{L}(\bz,v) := \left(\alpha_{\bm{\tau}}(\bz) + \beta_{\bm{\tau}}(\bz)v\right)^{2}. 
\label{integrand}  
\end{align}
The infimum in (\ref{TransformedLOMT}) is taken over all joint ``feedback linearized state" PDF and admissible (transformed) control pairs $(\sigma,v)\equiv(\sigma(\bz,t),v(\bz,t))$ satisfying (\ref{ConstrLiouvillez})-(\ref{ConstrPDFz}). The function $\sigma(\bz,t)$ is a weak solution of (\ref{ConstrLiouvillez}). Here, the set of admissible controls $\mathcal{V}:=\{v(\bz,t) \mid \bz\in\mathcal{Z}, t\in[0,1], \int_{\mathcal{Z}}v^{2}\differential\bz < \infty\}$. Problem (\ref{TransformedLOMT}) resembles optimal transport with linear prior dynamics (see e.g., \cite[equation (21)]{chen2017optimal}) with the exception that $\alpha_{\bm{\tau}}(\bz)\not\equiv 0$ and $\beta_{\bm{\tau}}(\bz)\not\equiv 1$, in general.

\begin{remark}\label{RemarkRecoverSoln}
If $(\sigma^{\rm{opt}}(\bz,t),v^{\rm{opt}}(\bz,t))$ is a solution for problem (\ref{TransformedLOMT}), then solution for problem (\ref{FLSprobMinEnergy}) (equivalently (\ref{FLSdensityprob})), denoted as $(\rho^{\rm{opt}}(\bx,t),u^{\rm{opt}}(\bx,t))$ can be recovered as
\begin{subequations}
\begin{align}
\rho^{\rm{opt}}(\bx,t) &= \sigma^{\rm{opt}}(\bm{\tau}(\bx),t)\:\lvert\det\left(\nabla_{\bx}\bm{\tau}(\bx)\right)\rvert,	\label{sigmaopt2rhoopt}\\
u^{\rm{opt}}(\bx,t) &= \alpha(\bx) + \beta(\bx) v^{\rm{opt}}(\bm{\tau}^{-1}(\bx),t), \label{vopt2uopt}
\end{align}
\label{RecoverFLSopt}	
\end{subequations}
for all $\bx\in\mathcal{X}$, $t\in[0,1]$.
\end{remark}

We now consider the feasibility of problem (\ref{TransformedLOMT}). Since $\bm{\tau}$ is a diffeomorphism, the map $\nabla\bm{\tau} : \mathcal{T}_{\bx}\mathcal{X} \mapsto \mathcal{T}_{\bz=\bm{\tau}(\bx)}\mathcal{Z}$ is an isomorphism \cite[Lemma 3.6d, pg. 55]{lee2013smooth}, and thanks to (\ref{AzbForm}), the flow  generated by (\ref{SingleInputControlAffine}) satisfying $\bx(t)\in\mathcal{X}$ implies that the flow generated by (\ref{feedbacklinearizedform}) satisfies $\bz(t)\in\mathcal{Z}$, i.e., ${\rm{spt}}(\sigma(\bz,t))\subseteq\mathcal{Z}$ for all $t\in[0,1]$ provided ${\rm{spt}}\left(\sigma_{0}\right),{\rm{spt}}\left(\sigma_{1}\right)\subseteq \mathcal{Z}$, which in turn is guaranteed due to \textbf{(A1)}. On the other hand, since the pair $(\bm{A},\bm{b})$ in (\ref{feedbacklinearizedform}) is controllable, hence any vector in $\mathcal{Z}$ is reachable from any other for $t\in[0,1]$. Therefore, in (\ref{ConstrPDFz}), the joint PDF $\sigma_{1}(\bz)$ is reachable from $\sigma_{0}(\bz)$ via the flow $\sigma(\bz,t)$ generated by the controlled Liouville PDE (\ref{ConstrLiouvillez}), which is indeed the ensemble flow associated with the trajectory flow $\bz(t)$ generated by the ODE (\ref{feedbacklinearizedform}). Thus, problem (\ref{TransformedLOMT}) is feasible. For details on the correspondence between the flow of an ODE and that generated by its associated Liouville PDE, see e.g., \cite{brockett2007optimal,halder2011dispersion,brockett2012notes}. Optimality issues are considered next.


\subsection{Optimality}\label{SubsecOptimality}
Since $\beta\neq 0$ (see Remark \ref{u2vandv2u}), hence from (\ref{ab}), $\beta_{\bm{\tau}}\neq 0$. 
Let $m:=v\sigma$ and consider the change of variable $(\sigma,v)\mapsto (\sigma,m)$, transforming (\ref{TransformedLOMT}) as
\begin{subequations}
\begin{align}
&\qquad\;\underset{(\sigma,m)}{\text{inf}}
& & \int_{\mathcal{Z}} \int_{0}^{1} J(\bz,\sigma,m)\: \differential t \: \differential\bz\label{AgainObjFluidz}\\ 
& \text{subject to} & &  \dfrac{\partial \sigma}{\partial t} + \nabla_{\bz} \cdot (\bA\bz\sigma + \bb m) = 0,\label{AgainConstrLiouvillez}\\
& & &  \sigma(\bz,0)=\sigma_{0}(\bz),  \quad \sigma(\bz,1)=\sigma_{1}(\bz),\label{AgainConstrPDFz}
\end{align}
\label{AgainTransformedLOMT} 
\end{subequations}
where 
\begin{align}
J(\bz,\sigma,m):= \begin{cases}
 \frac{1}{2}\left(\alpha_{\bm{\tau}}(\bz) + \beta_{\bm{\tau}}(\bz)\frac{m}{\sigma}\right)^{2} \sigma &\text{if}\quad\sigma >0,\\
 0 &\text{if}\quad(\sigma,m)=(0,0),\\
 +\infty &\text{otherwise}.	
 \end{cases}
\label{defJsigmam}	
\end{align}
Since the map $m \mapsto \frac{1}{2}\left(\alpha_{\bm{\tau}}(\bz) + \beta_{\bm{\tau}}(\bz)m\right)^{2}$ is convex in $m$, its perspective function $J(\bz,\sigma,m)$ is jointly convex in $(\sigma,m)$, and is lower semi-continuous. The constraints (\ref{AgainConstrLiouvillez})-(\ref{AgainConstrPDFz}) are linear in $(\sigma,m)$. Therefore, problem (\ref{AgainTransformedLOMT}) is convex in $(\sigma,m)$. Yet, the existence-uniqueness of minimizer $(\sigma^{\rm{opt}},v^{\rm{opt}})$ for (\ref{TransformedLOMT}), or equivalently for (\ref{AgainTransformedLOMT}), is not obvious. To establish the same, we proceed in three steps.

\noindent\underline{Step 1 (finding the static optimal transport problem induced by (\ref{integrand})):} Consider the cost $c : \mathcal{Z}\times\mathcal{Z}\mapsto\mathbb{R}_{\geq 0}$ given by
\begin{align}
c(\bz_0,\bz_1) := &\underset{v\in\mathcal{V}}{\inf} \displaystyle\int_{0}^{1}\frac{1}{2}\mathcal{L}\left(\bz,v\right)\:\differential t\nonumber\\
&\text{subject to}	\quad \dot{\bz}=\bA\bz + \bb v, \nonumber\\
&\qquad\qquad\quad\bz(t=0) = \bz_0, \bz(t=1)=\bz_1.
\label{staticMKcost}	
\end{align}
Thanks to the special structure for the pair $(\bA,\bb)$ from (\ref{feedbacklinearizedform}), the pseudo-inverse of $\bb$ equals $\be_{n}^{\top}$, and it follows that (\ref{staticMKcost}) can be expressed in terms of an action integral
\begin{align}
\!\!\!\!c(\bz_0,\bz_1) = &\underset{\bz(\cdot)\in\{H^{1,1}\left([0,1]\right)\mid \bz(0)=\bz_0,\bz(1)=\bz_1\}}{\inf}\!\displaystyle\int_{0}^{1}\!\widetilde{\mathcal{L}}\left(\bz(t),\dot{\bz}(t)\right)\:\differential t,
\label{ActionIntegral}	
\end{align}
where the Sobolev space $H^{1,1}\left([0,1]\right):=\{\bz(t)\in L^{1}([0,1]) \mid \text{first order weak derivative of}\;\bz(t)\in L^{1}([0,1])\}$. Clearly, $\bz\in H^{1,1}\left([0,1]\right)$ is absolutely continuous. In (\ref{ActionIntegral}), the autonomous Lagrangian $\widetilde{\mathcal{L}}:\mathcal{T}\mathcal{Z}\mapsto\mathbb{R}_{\geq 0}$, and is given by
\begin{align}
\widetilde{\mathcal{L}}\left(\bz,\dot{\bz}\right) := \frac{1}{2}\left(\alpha_{\bm{\tau}}(\bz) + \beta_{\bm{\tau}}(\bz)\langle\bm{e}_{n},\dot{\bz}\rangle\right)^{2}.
\label{TonelliLagrangian}	
\end{align}
With (\ref{ActionIntegral}), we associate a static optimal transport problem
\begin{align}
\underset{\stackrel{T:\mathcal{Z}\mapsto\mathcal{Z}}{\sigma_{1}=T\sharp\sigma_0}}{\inf}\int_{\mathcal{Z}}c\left(\bz_0,T\left(\bz_0\right)\right)\sigma_{0}(\bz_0)\differential\bz_0,
\label{OMTfromActionIntegral}	
\end{align}
and its Kantorovich (i.e., optimal coupling) version
\begin{subequations}
\begin{align}
& \qquad\;\underset{\pi}{\text{inf}}& & \!\!\!\! \int_{\mathcal{Z} \times \mathcal{Z} }  c(\bz_0,\bz_1) \pi( \: \differential\bz_{0} \: \differential \bz_{1})\\ 
& \text{subject to} & &	 \pi(\differential \bz_{0}\times \mathcal{Z})=\sigma_0(\bz_{0})\differential\bz_{0} ,\label{piconstr1}\\
& & & \pi( \mathcal{Z}\times \differential\bz_{1} )=\sigma_1(\bz_{1})\differential\bz_{1},\label{piconstr2}
\end{align}
\label{KantorovichfromActionIntegral}	
\end{subequations}
where the $\arg\inf$ in (\ref{OMTfromActionIntegral}) is the optimal transport map $T^{\rm{opt}}$, and the $\arg\inf$ in (\ref{KantorovichfromActionIntegral}) is the optimal transport plan $\pi^{\rm{opt}}$. We will soon relate the above problems with (\ref{TransformedLOMT}) or equivalently with (\ref{AgainTransformedLOMT}). Before doing so, we argue that (\ref{OMTfromActionIntegral}), (\ref{KantorovichfromActionIntegral}) admit unique solution.

From (\ref{TonelliLagrangian}), we note that $\widetilde{\mathcal{L}}(\bz,\cdot)$ is convex but not strictly convex for each $\bz\in\mathcal{Z}$. So $\widetilde{\mathcal{L}}$ is not a (weak) Tonelli Lagrangian \cite[Ch. 6.2]{figalli2007optimal}; see also \cite{bernard2007optimal}, \cite[p. 118]{villani2008optimal}. Consequently, we cannot invoke results such as \cite[Theorem 1.4.2]{figalli2007optimal} to conclude unique minimizers for (\ref{OMTfromActionIntegral}) or (\ref{KantorovichfromActionIntegral}).

For $\bz\in\mathcal{Z}\subseteq\mathbb{R}^{n}$ denote the $n$-th component of $\dot{\bz}$ as $\dot{z}_{n}$, and notice that the second argument of $\widetilde{\mathcal{L}}$ in (\ref{TonelliLagrangian}) only depends on $\dot{z}_{n}$, i.e., $\widetilde{\mathcal{L}}(\bz,\dot{\bz})\equiv\widetilde{\mathcal{L}}(\bz,\dot{z}_{n})$. As $\lim_{|\dot{z}_{n}|\rightarrow\infty}\widetilde{\mathcal{L}}(\bz,\dot{z}_{n})/|\dot{z}_{n}| = +\infty$, the Lagrangian $\widetilde{\mathcal{L}}$ is superlinear (i.e., 1-coercive) in $\dot{z}_{n}$. Also, since $\alpha,\beta$ in (\ref{ComputeAlpha}),(\ref{ComputeBeta}) are continuous (which follows from $\bm{f},\bm{g}$ being componentwise of class $C^{n+1}(\mathcal{X})$; see Sec. \ref{SubsecStochasticControlProblem}), and $\bm{\tau}$ is a diffeomorphism, both
$\widetilde{\mathcal{L}}(\bz,\dot{\bz})$ and its derivative w.r.t. $\dot{\bz}$ are continuous in $\mathcal{TZ}$. Then, by \cite[Theorem 3.7]{buttazzo1998one}, the minimizer $\bz^{{\rm{opt}}}(\cdot)$ in (\ref{ActionIntegral}) exists in $H^{1,1}([0,1])$. Having the existence in $H^{1,1}([0,1])$, it can be shown \cite[Theorem 16.8]{clarke2013functional} a posteriori that the minimizer in (\ref{ActionIntegral}) has better regularity\footnote{The conditions in \cite[Theorem 16.8]{clarke2013functional} are satisfied by our $\widetilde{\mathcal{L}}$ since it is continuous, convex in $\dot{\bz}$, has no explicit $t$-dependence, and being superlinear in $\dot{z}_{n}$, satisfies the so-called Nagumo growth condition \cite[p. 329]{clarke2013functional}.}: it is in fact Lipschitz continuous in $t\in[0,1]$. Consequently, the minimizing pair $(\bz^{{\rm{opt}}},v^{{\rm{opt}}})$ for (\ref{staticMKcost}) exists ($v^{{\rm{opt}}} = \bm{e}_{n}^{\top}(\dot{\bz}^{{\rm{opt}}}-\bm{A}\bz^{{\rm{opt}}})$). For each feasible $v\in\mathcal{V}$, the Hamiltonian $\mathcal{H}_{v} : (\bz,\bm{p})\in\mathcal{T}^{*}\mathcal{Z}\mapsto\mathbb{R}$ for (\ref{staticMKcost}) is
\[\mathcal{H}_{v}(\bz,\bm{p}) = \frac{1}{2}\left(\alpha_{\bm{\tau}}(\bz) + \beta_{\bm{\tau}}(\bz)v\right)^{2} + \sum_{i=1}^{n-1}p_{i}z_{i+1} + p_{n}v,\]
and applying Pontryagin's maximum principle, we get the two point boundary value problem associated with Hamiltonian vector field 
\begin{align*}
\dot{\bz}^{\rm{opt}} &= \left(z_2, z_3, z_{4}, \hdots, z_{n}, -\frac{p_{n}}{\beta_{\bm{\tau}}^{2}}-\frac{\alpha_{\bm{\tau}}}{\beta_{\bm{\tau}}}\right)^{\!\!\top}\bigg\vert_{(\bz^{\rm{opt}},\bm{p}^{\rm{opt}})},\\
\dot{\bm{p}}^{\rm{opt}} &= \frac{p_{n}^{2}}{\beta_{\bm{\tau}}^{3}}\nabla_{\bz}\beta_{\bm{\tau}} + \begin{pmatrix}
 	0, p_{1}, p_{2}, \hdots, p_{n-2}, p_{n-1}
 \end{pmatrix}^{\!\!\top}\bigg\vert_{(\bz^{\rm{opt}},\bm{p}^{\rm{opt}})},
\end{align*}
with boundary conditions $\bz^{\rm{opt}}(0)=\bz_{0},\bz^{\rm{opt}}(1)=\bz_{1}$, resulting in the optimal cost $c(\bz_0,\bz_1)=\frac{1}{2}\int_{0}^{1}\left(\frac{p_{n}^{\rm{opt}}(t)}{\beta_{\bm{\tau}}(\bz^{\rm{opt}}(t))}\right)^{2} \differential t$. Since $\mathcal{H}_{v}$ is $C^2$, and the Hamiltonian vector field is complete in $\mathcal{T}^{*}\mathcal{Z}$, the conditions in \cite[Theorems 5.3, 5.7, 4.1]{agrachev2009optimal} are satisfied, and allow us to conclude the existence and uniqueness of the optimal transport map $T^{\rm{opt}}$ in (\ref{OMTfromActionIntegral}), and that of the optimal transport plan $\pi^{\rm{opt}}$ in (\ref{KantorovichfromActionIntegral})--the latter being concentrated on the graph of $T^{\rm{opt}}$.


\noindent\underline{Step 2 (lower bounding (\ref{TransformedLOMT})):} Denote the objective in (\ref{ObjFluidz}) by $\Xi(\sigma,v)$, and the optimal value of (\ref{OMTfromActionIntegral}) (or equivalently of (\ref{KantorovichfromActionIntegral})) as ${\wp}^{{\rm{opt}}}(\sigma_0,\sigma_1)$. For any feasible $v\in\mathcal{V}$, let $\mathcal{S}_{01}$ be the set of all PDF trajectories $\sigma(\bz,t)$ with fixed endpoints $\sigma_{0}$ and $\sigma_{1}$, i.e., 
\begin{align}
\mathcal{S}_{01}:=\{\sigma(\bz,t)\mid\sigma\geq 0,\int_{\mathcal{Z}}\sigma\differential\bz = 1, \; \text{(\ref{ConstrPDFz}) holds}\}. 
\label{defP01}	
\end{align}
Let $\mathcal{P}_{01}:=\{\sigma(\bz,t)\differential\bz\mid t\in[0,1],\sigma(\bz,t)\in\mathcal{S}_{01}\}$ denote the collection of probability measures on the path space $C[0,1]$ induced by the controlled sample path $\bz(\cdot)$ whose marginals at $t=0,1$ are $\sigma_{0}\differential\bz,\sigma_{1}\differential\bz$, respectively. Likewise, let $\mathcal{D}_{01}$ be the collection of probability measures on $C[0,1]$ paths of the controlled state $\bz(\cdot)$ whose marginals at $t=0,1$ are Dirac deltas $\delta_{\bz_{0}},\delta_{\bz_{1}}$, concentrated at $\bz_{0},\bz_{1}$, respectively. Then, the probability measure on $\{\bz(\cdot)\in H^{1,1}([0,1])\mid \bz(0)=\bz_{0},\bz(1)=\bz_{1}\}$ concentrated on the optimal path $\bz^{{\rm{opt}}}(\cdot)$ in (\ref{ActionIntegral}), solves 
\begin{align}
\underset{D_{01}\in\mathcal{D}_{01}}{\inf} \mathbb{E}_{D_{01}}\bigg\{\displaystyle\int_{0}^{1}\widetilde{\mathcal{L}}\left(\bz(t),\dot{\bz}(t)\right)\:\differential t\bigg\}.
\label{DiracCOV}	
\end{align}
Since any $P_{01}\in\mathcal{P}_{01}$ allows the disintegration of measure: $P_{01} = \int_{\mathcal{Z} \times \mathcal{Z}}D_{01}\:\pi\left(\differential\bz_{0}\differential\bz_{1}\right)$, where $\pi$ satisfies (\ref{piconstr1})-(\ref{piconstr2}), we have
\begin{align*}
\underset{(\sigma,v)\in\mathcal{S}_{01}\times\mathcal{V}}{\inf}\Xi(\sigma,v) &\geq \mathbb{E}_{P_{01}}\bigg\{\displaystyle\int_{0}^{1}\widetilde{\mathcal{L}}\left(\bz(t),\dot{\bz}(t)\right)\:\differential t\bigg\}\nonumber\\
&= \displaystyle\int_{\mathcal{Z}\times\mathcal{Z}}\!\!\mathbb{E}_{D_{01}}\bigg\{\!\!\displaystyle\int_{0}^{1}\!\!\widetilde{\mathcal{L}}\left(\bz(t),\dot{\bz}(t)\right)\:\differential t\bigg\}\pi\left(\differential\bz_{0}\differential\bz_{1}\right)	 \nonumber\\
&\geq \displaystyle\int_{\mathcal{Z}\times\mathcal{Z}} c(\bz_{0},\bz_{1}) \pi\left(\differential\bz_{0}\differential\bz_{1}\right).
\end{align*}
Since the above holds for any feasible transport plan $\pi$, we deduce 
\begin{align}
	\underset{(\sigma,v)\in\mathcal{S}_{01}\times\mathcal{V}}{\inf}\Xi(\sigma,v) \geq {\wp}^{{\rm{opt}}}(\sigma_0,\sigma_1).
\label{LOWERbound}	
\end{align}
From Step 1, we know that ${\wp}^{{\rm{opt}}}$ is finite, and is achieved by a unique optimal transport map $T^{\rm{opt}}:\mathcal{Z}\mapsto\mathcal{Z}$.

\noindent\underline{Step 3 (achieving equality in (\ref{LOWERbound})):} For $t\in[0,1]$, the map $T^{\rm{opt}}_{t}(\bz):=(1-t)\bz + tT^{\rm{opt}}(\bz)$ induces a displacement interpolation $\sigma_{t}:=T^{\rm{opt}}_{t}(\bz)\sharp\sigma_{0}$. Let ${\rm{Id}}$ denote the identity map, and associate a vector field $\bm{v}_{t}\in L^{2}(L^{\infty}(\mathcal{Z});[0,1])$ given by $\bm{v}_{t}:=(T^{\rm{opt}} - {\rm{Id}})\circ \left(T^{\rm{opt}}_{t}\right)^{-1}$, defined almost everywhere w.r.t. $\sigma_t\differential\bz$ with the convention that $\sigma_{t}=0$ whenever $\bm{v}_{t}=0$. As is well-known from the theory of optimal mass transport \cite{figalli2010mass}, the map $T^{\rm{opt}}$ is the exponential map of the gradient of a convex function, hence $T^{\rm{opt}}_{t}(\bz)$ is injective, and the vector field $\bm{v}_{t}$ is well-defined. Then, proceeding as in \cite[p. 242]{villani2003topics}, the pair $(\sigma_t,\bm{v}_t)$ solves the continuity equation $\frac{\partial\sigma_t}{\partial t} + \nabla_{\bz}\cdot(\sigma_t\bm{v}_t) = 0$ in the weak sense. Equivalently, the pair $(\sigma := \sigma_t, v := \bm{e}_{n}^{\top}(\bm{v}_t - \bm{Az}))$ solves (\ref{ConstrLiouvillez}) in the weak sense, satisfies (\ref{ConstrPDFz}), and $\Xi(\sigma,v) = \int_{\mathcal{Z}}\int_{0}^{1}\frac{1}{2}\mathcal{L}(\bz,v)\differential t\sigma\differential\bz = \int_{\mathcal{Z}} \sigma_{0}(\bz)c(\bz,T^{\rm{opt}}(\bz))\differential\bz = \wp^{\rm{opt}}(\sigma_0,\sigma_1)$. This guarantees existence-uniqueness of a pair $(\sigma,v)$ that achieves equality in (\ref{LOWERbound}). We refer to this pair as $(\sigma^{\rm{opt}},v^{\rm{opt}})$.

We next consider the Lagrangian associated with (\ref{TransformedLOMT}) given by
\begin{align}
\mathscr{L}(\sigma,\psi,v) := &\int_{\mathcal{Z}} \int_{0}^{1}  \bigg\{\frac{1}{2} \mathcal{L}(\bz,v) \sigma(\bz,t)  + \psi(\bz,t)  							\times\nonumber\\
                                            & \bigg( \dfrac{\partial \sigma}{\partial t} +\nabla_{\bz} \cdot\left((\bA \bz +  \bb v )\sigma\right) \bigg) 							\differential t \: \differential\bz\bigg\},
\label{Lagrangian}
\end{align}
where $\psi\left(\bz,t\right)$ is the Lagrange multiplier or costate. Then the optimal control $v^{\rm{opt}}(\bz,t)$ in (\ref{TransformedLOMT}) can be obtained by performing the unconstrained minimization of (\ref{Lagrangian}) over $(\sigma,v)\in\mathcal{S}_{01}\times\mathcal{V}$. We summarize the result in the Theorem below.

\begin{theorem}\label{ThmHJBopt}(\textbf{Optimal control for (\ref{TransformedLOMT})})
For $\bz\in\mathcal{Z}$ and $t\in[0,1]$, the optimal control $v^{\rm{opt}}(\bz,t)$ in (\ref{TransformedLOMT}) is
\begin{align}
v^{\rm{opt}}(\bz,t) = \frac{1}{\beta_{\bm{\tau}}^{2}(\bz)}\frac{\partial}{\partial z_{n}} \psi(\bz,t) - \frac{\alpha_{\bm{\tau}}(\bm{z})}{\beta_{\bm{\tau}}(\bm{z})},
\label{voptOptimal}	 
\end{align}
where $\psi(\bm{z},t)$ solves the Hamilton-Jacobi-Bellman (HJB) PDE
\begin{align}
\!\!\!\frac{\partial\psi}{\partial t} + \!\displaystyle\sum_{i=1}^{n-1}\!z_{i+1}\frac{\partial\psi}{\partial z_{i}} - \frac{\alpha_{\bm{\tau}}(\bm{z})}{\beta_{\bm{\tau}}(\bm{z})}\frac{\partial\psi}{\partial z_{n}} + \frac{1}{2\beta_{\bm{\tau}}^{2}(\bz)}\!\left(\!\frac{\partial\psi}{\partial z_{n}}\!\right)^{\!\!2}\!\!=\! 0,
\label{HJBoptimal}	
\end{align}
subject to suitable terminal condition $\psi(\bz,1)\equiv\psi_{1}(\bz)$.	
\end{theorem}
\begin{proof}
See Appendix \ref{appendix:deriveHJB}.	
\end{proof}
\noindent For a generalization of the above result in the vector input case, we refer the readers to \cite[Theorem 1]{caluya2019finite}.
\begin{remark}(\textbf{Optimal controlled PDF for (\ref{TransformedLOMT})})
Given a solution $\psi(\bz,t)$ of (\ref{HJBoptimal}), and $v^{\rm{opt}}(\bz,t)$ as in (\ref{voptOptimal}), if the solution of the Liouville PDE initial value problem
\begin{subequations}
	\begin{align}
\dfrac{\partial \sigma^{\rm{opt}}}{\partial t} + \nabla_{\bz} \cdot ((\bA\bz + \bb v^{\rm{opt}}(\bz,t) )\sigma^{\rm{opt}}) = 0, \label{LiouvilleOptimalPDE}\\ 
\sigma^{\rm{opt}}(\bz,0)=\sigma_{0}(\bz),\label{LiouvilleOptimalIC}	
\end{align}
\label{LiouvilleOptimalIVP}
\end{subequations}
satisfies $\sigma^{\rm{opt}}(\bz,1)=\sigma_{1}(\bz)$, then the pair $(\sigma^{\rm{opt}},v^{\rm{opt}})$ solves problem (\ref{TransformedLOMT}).	
\end{remark}

In general, the HJB PDE may only admit viscosity solution \cite{crandall1983viscosity,fleming2006controlled}. However, noting that (\ref{HJBoptimal}) is of the specific form 
\begin{align}
\frac{\partial\psi}{\partial t} + \mathcal{H}\left(\bz,\nabla_{\bz}\psi\right)=0,
\label{HJBform}	
\end{align}
wherein for $\bz,\bm{\zeta}\in\mathbb{R}^{n}$, the state-dependent Hamiltonian 
\begin{align}
\!\!\mathcal{H}\left(\bz,\bm{\zeta}\right)\!:=\!\!\displaystyle\sum_{i=1}^{n-1}\!z_{i+1}\zeta_{i} - \frac{\alpha_{\bm{\tau}}(\bm{z})}{\beta_{\bm{\tau}}(\bm{z})}\zeta_{n} + \frac{1}{2\beta_{\bm{\tau}}^{2}(\bz)}\zeta_{n}^{2},
\label{HJBHamiltonian}	
\end{align}
we aim to find representation formula for its solution. 

\subsubsection{Envelope representation for $\psi$}\label{subsubEnv}  
Theorem \ref{ThmHJBoptHopfLax} below gives the so-called ``upper and lower envelope representation formula" for $\psi(\bz,t)$ at the expense of additional assumptions on the vector fields $\bm{f},\bm{g}$. Such formula, when possible to derive, allow representing $\psi$ as pointwise $\inf$ (for upper envelope), or pointwise $\sup$ (for lower envelope) of elementary functions, thereby bypassing the general but computationally challenging viscosity approach (regularizing the HJB PDE by the Laplacian $\Delta\psi$, then passing to the limit, see e.g., \cite[Section 10.1]{evans1998graduate}). In the special case of \emph{state-independent} Hamiltonian (i.e., $\mathcal{H}(\bm{z},\bm{\zeta})\equiv\mathcal{H}(\bm{\zeta})$), such variational representations reduce to the well-known Hopf-Lax formula \cite{hopf1965generalized,lax1957hyperbolic}. That the state-dependent Hamiltonian such as (\ref{HJBHamiltonian}) may still be amenable for envelope representation using convex conjugates \cite{rockafellar2000convexity}, seem to be less known in the control community. The following result builds on Appendix \ref{appendix:EnvelopeFormula}.
 
\begin{theorem}\label{ThmHJBoptHopfLax}(\textbf{Envelope formula for the solution of (\ref{HJBoptimal})})
Given the vector fields $\bm{f},\bm{g}$ in (\ref{FLSprobMinEnergy}) (or equivalently in (\ref{FLSdensityprob})), let the maps $\alpha_{\bm{\tau}}(\cdot),\beta_{\bm{\tau}}(\cdot)$ be given by (\ref{alphabetatau}) and (\ref{ab}). Suppose that $\bm{f},\bm{g}$ are such that (\ref{HJBHamiltonian}) is concave in $\bz\in\mathcal{Z}$. Let 
\begin{align}
\bm{a}(\bz):=\bA\bz + L_{\bm{f}}^{n}\lambda(\bm{\tau}^{-1}(\bz))\bb,
\label{defa}	
\end{align}
and
\begin{align}
\ell(\bz,\bm{w}) := \begin{cases}
\!\!\frac{\left(\bm{w}-\bm{a}(\bz)\right)^{\top}\left(\bb\bb^{\top}\right)^{\dagger}\left(\bm{w}-\bm{a}(\bz)\right)}{2\left(L_{\bm{g}}L_{\bm{f}}^{n-1}\lambda(\bm{\tau}^{-1}(\bz))\right)^{2}}, \!\!&\text{if}\;\bm{w}-\bm{a}(\bm{z})\in\mathcal{R}\left(\bb\bb^{\top}\right),\\
\!+\infty, &\text{otherwise,} 	
 \end{cases}
\label{TrueLagrangian}	
\end{align}
where $^{\dagger}$ denotes the pseudo-inverse, $\mathcal{R}(\cdot)$ denotes the range. Let the dualizing Kernel $K$ associated with (\ref{TrueLagrangian}) be given by
\begin{align}
	\!\!K(t,\bz,\bm{r})\! := \underset{\bm{q}(t)}{\inf}\bigg\{\!\langle\bm{q}(0),\bm{r}\rangle \!+ \!\!\int_{0}^{t}\!\!\!\ell\left(\bm{q}(t),\dot{\bm{q}}(t)\right)\differential t \mid \bm{q}(t)=\bz\!\bigg\},
\label{defK}	
\end{align}
with $K(0,\bz,\bm{r}):= \langle\bz,\bm{r}\rangle$. Then the upper and lower envelope representation for $\psi(\bz,t)$ in (\ref{HJBform}) subject to the initial condition $\psi(\bz,0)=\psi_{0}(\bz)$, are given by the formula (\ref{GeneralUpperEnvelope}) and (\ref{GeneralLowerEnvelope}) in Appendix \ref{appendix:EnvelopeFormula}, respectively.
\end{theorem}
\begin{proof}
See Appendix \ref{appendix:DeriveEnvelope}.	
\end{proof}

\subsubsection{Characteristic representation for $\psi$}\label{subsubChar} 
A different representation for $\psi(\bz,t)$ can be obtained by deriving the characteristic ODEs associated with the first order nonlinear PDE (\ref{HJBform}), as summarized next.

\begin{theorem}\label{ThmJHBoptMOC}(\textbf{Characteristic equations for (\ref{HJBoptimal})})
Consider the pair $(\bA,\bb)$ as in (\ref{feedbacklinearizedform}), and let $\theta:=\frac{\partial\psi}{\partial t}$. The characteristic equations governing the evolution of the parametric strips\footnote{In the case of quasilinear PDE, the characteristic equations describe characteristic ``curves". For a nonlinear PDE like (\ref{HJBform}), the terminology characteristic ``strips" is used since equations (\ref{dzetads}) and (\ref{dthetads}), appearing because of the nonlinearity in the PDE, describe evolution of elements on tangent space.} $\left(t(s),\bz(s),\bm{\zeta}(s),\theta(s),\psi(s)\right)$ associated with (\ref{HJBform})-(\ref{HJBHamiltonian}) are:
\begin{subequations} 
    \begin{align}
      \frac{\differential t}{\differential s} & = 1, \label{dtds}\\
       \frac{\differential \bm{z}}{\differential s} &=  \nabla_{\bm{\zeta}} \mathcal{H}(\bm{z},\bm{\zeta}) = \bA \bm{z} + \frac{\bb\bb^{\top} \bm{\zeta}}{\beta_{\bm{\tau}}^2(\bz)} -\frac{\alpha_{\bm{\tau}}(\bz)}{\beta_{\bm{\tau}}(\bz)} \bb, \label{dzds}\\
        \frac{\differential \bm{\zeta}}{\differential s} &=  -\nabla_{\bm{z}} \mathcal{H}(\bm{z},\bm{\zeta}) = -\bA^{\top} \bm{\zeta}+ \frac{\bb^{\top} \bm{\zeta}   }{\beta_{\bm{\tau}}^2(\bz)}\bigg(\!\beta_{\bm{\tau}}(\bz)\nabla_{\bz} \alpha_{\bm{\tau}}(\bz)  \nonumber\\
        &- \alpha_{\bm{\tau}}(\bz) \nabla_{\bz} \beta_{\bm{\tau}}(\bz) \!\bigg) + \frac{(\bb^{\top} \bm{\zeta})^2}{\beta_{\bm{\tau}}^3(\bz)} \nabla_{\bz} \beta_{\bm{\tau}}(\bz),  \label{dzetads}\\
        \frac{\differential \theta} {\differential s} & = 0, \label{dthetads}\\
        \frac{\differential \psi} {\differential s} &=  \langle\nabla_{\bm{\zeta}}\mathcal{H}(\bm{z},\bm{\zeta}),\bm{\zeta}\rangle - \mathcal{H}(\bm{z},\bm{\zeta}) = \frac{(\bb^{\top} \bm{\zeta})^2}{2\beta_{\bm{\tau}}^2(\bz)}. \label{dpsids}      
    \end{align}
    \label{CharODE}
\end{subequations}
\end{theorem}
\begin{proof}
Transcribing (\ref{HJBform}) in the form 
\[F(t,\bz,\bm{\zeta},\theta,\psi) := \theta + \langle \bA\bz,\bm{\zeta}\rangle - \frac{\alpha_{\bm{\tau}}(\bz)}{\beta_{\bm{\tau}}(\bz)} \langle\bb,\bm{\zeta}\rangle + \frac{\langle\bb,\bm{\zeta}\rangle^2}{2\beta_{\bm{\tau}}^2(\bz)}=0,\]
and then writing the standard Lagrange-Charpit equations: 
\begin{align*}
&\frac{\differential t}{\differential s} = \frac{\partial F}{\partial\theta}, \quad \frac{\differential \bm{z}}{\differential s} = \frac{\partial F}{\partial\bm{\zeta}}, \quad \frac{\differential \theta} {\differential s} = -\frac{\partial F}{\partial t} - \theta\frac{\partial F}{\partial\psi}, \\
&\frac{\differential \bm{\zeta}}{\differential s} = -\frac{\partial F}{\partial\bz} -\bm{\zeta}\frac{\partial F}{\partial\psi}, \quad \frac{\differential \psi} {\differential s} = \theta\frac{\partial F}{\partial\theta} + \bigg\langle\bm{\zeta},\frac{\partial F}{\partial\bm{\zeta}}\bigg\rangle,	
\end{align*}
results in the system of $2n+3$ coupled nonlinear ODEs (\ref{CharODE}). 	
\end{proof}
Since $s=t$ from (\ref{dtds}), the characteristic strips are obtained by integrating (\ref{dzds})--(\ref{dpsids}) for $s=t\in[0,1]$ with prescribed boundary condition $\psi(\bz,t=1)=\psi_{1}(\bz)$. While the system of ODEs (\ref{CharODE}) is quite nonlinear to allow for  closed-form solution, we will see in Section \ref{subsubCharApprox} that a special case of the same is amenable for analytical treatment.
 
%
%


\section{Feasible Steering}\label{SecSuboptimal}
In this Section, we derive a feasible controller that satisfies the constraints (\ref{FLSdyn})-(\ref{PDFconstr2pbvp}). To that end, we give a modification of the minimum energy PDF steering problem (\ref{FLSprobMinEnergy}) (equivalently, problem (\ref{FLSdensityprob}) or (\ref{TransformedLOMT})) by considering the problem below in the feedback linearized coordinates, given by 
\begin{subequations}
\begin{align}
&\qquad\;\underset{(\widetilde{\sigma},\widetilde{v})}{\text{inf}}
& & \int_{\mathcal{Z}} \int_{0}^{1} \frac{1} {2}   \lvert \widetilde{v}(\bz,t) \rvert^{2} \,\widetilde{\sigma}(\bz,t) \: \differential t \: \differential\bz\label{ObjFluidzRelax}\\ 
& \text{subject to} & &  \dfrac{\partial \widetilde{\sigma}}{\partial t} + \nabla_{\bz} \cdot ((\bA\bz + \bb \widetilde{v} )\widetilde{\sigma}) = 0,\label{ConstrLiouvillezRelax}\\
& & &  \widetilde{\sigma}(\bz,0)=\sigma_{0}(\bz),  \quad \widetilde{\sigma}(\bz,1)=\sigma_{1}(\bz),\label{ConstrPDFzRelax}
\end{align}
\label{TransOCFluidsrelax} 
\end{subequations}
which differs from (\ref{TransformedLOMT}) only in the objective. Here, $\bz=\bm{\tau}(x)$, $u=\alpha(\bx)+\beta(\bx)\widetilde{v}$, as before. We will see that this modified problem is slightly more tractable than the problem formulated in Section \ref{SectionMinEnergySteering}. The existence and uniqueness for the minimizer $(\widetilde{\sigma}^{\text{opt}},\widetilde{v}^{\text{opt}})$ for  (\ref{TransOCFluidsrelax}) can be guaranteed by following arguments similar to Section \ref{SubsecOptimality}.

Unlike the optimal pair $(\sigma^{\text{opt}},v^{\text{opt}})$ for problem (\ref{TransformedLOMT}), the optimal pair $(\widetilde{\sigma}^{\text{opt}},\widetilde{v}^{\text{opt}})$ for problem (\ref{TransOCFluidsrelax}) no longer corresponds to the \emph{minimum energy} optimal transport solution $(\rho^{\text{opt}},u^{\text{opt}})$ for problem (\ref{FLSdensityprob}), and is only a feasible solution that guarantees steering the prescribed $\rho_{0}$ to $\rho_{1}$ via the given controlled dynamics. Indeed, we can rewrite (\ref{TransOCFluidsrelax}) in the physically meaningful variables $(\rho,u)$ as  
\begin{subequations}
\begin{align}
& \qquad\;\underset{(\rho,u)}{\text{inf}}
& & \!\!\!\!\int_{\mathcal{X}} \!\int_{0}^{1} \!\frac{1} {2}  \left(\! \frac{u(\bx,t) - \alpha(\bx)}{\beta(\bx)} \!\right)^{2}\!\rho(\bx,t) \: \differential t \: \differential \bx\label{OrigOCFluidsRelax}\\ 
& \text{subject to} & &  \dfrac{\partial \rho}{\partial t} + \nabla_{\bx} \cdot \left(\left(\bm{f}(\bx) + \bm{g}(\bx)u\right) \rho\right) = 0, \label{FLSLiouvilleRelax}\\
& & &  \rho(\bx,0)= \rho_{0}(\bx),\quad\rho(\bx,1) = \rho_{1}(\bx).\label{FLSpdftpbvpRelax}
\end{align} 
\label{FLSdensityprobRelax}
\end{subequations}
Nonetheless, solving (\ref{TransOCFluidsrelax}) (equivalently, (\ref{FLSdensityprobRelax})) is of interest since it furnishes a controller $u(\bx,t)$ for steering any prescribed $\rho_{0}$ to any other prescribed $\rho_{1}$ in unit time subject to (\ref{SingleInputControlAffine}). In the following, we focus on the solution of (\ref{TransOCFluidsrelax}).

\subsection{Solution of (\ref{TransOCFluidsrelax}) via Optimal Transport Map}\label{SubsecOMTapprox}

In \cite{chen2017optimal}, the dynamic problem (\ref{TransOCFluidsrelax}) was shown to be equivalent to a static Mong\'{e}-Kantorovich optimal transport problem of the form
\begin{subequations}
\begin{align}
& \qquad\;\underset{\hat{\pi}}{\text{inf}}& & \!\!\!\! \int_{\mathcal{Z} \times \mathcal{Z} }  \frac{1} {2}  \lVert \hat{\bz}_{1} - \hat{\bz}_{0} \rVert_{2}^2 \:\hat{\pi}( \: \differential \hat{\bz}_{0} \: \differential \hat{\bz}_{1})\\ 
& \text{subject to} & &	 \hat{\pi}(\differential \hat{\bz}_{0}\times \mathcal{Z})=\hat{\sigma}_0(\hat{\bz}_{0})\differential\hat{\bz}_{0} ,\\
& & & \hat{\pi}( \mathcal{Z}\times \differential \hat{\bz}_{1} )=\hat{\sigma}_1(\hat{\bz}_{1})\differential \hat{\bz}_{1},
\end{align}
\label{OMTrelax}	
\end{subequations}
where the infimum is taken over all probability measures $\hat{\pi}$ supported on the product space $\mathcal{Z}\times\mathcal{Z}$ with marginals $\hat{\sigma}_0(\hat{\bz}_{0})\differential \hat{\bz}_{0}$ and $\hat{\sigma}_1(\hat{\bz}_{1})\differential \hat{\bz}_{1}$, given by\footnote{The corresponding formula for (\ref{OMTrelaxMarg0}) in \cite[p. 2141, unnumbered eqn. after (28)]{chen2017optimal} contains an extra factor $\lvert\det(\bm{\Phi}_{10})\rvert^{-1}$, which in our case equals unity; see (\ref{STMexplicit}) in Appendix \ref{appendix:STMGramianFormula}.}
\begin{subequations}
\begin{align}
	    \hat{\sigma}_0(\hat{\bz}_{0}) &:= \det\left(\bm{M}_{10}\right)^{\frac{1}{2}}
    \sigma_0(\bm{\Phi}_{10}^{-1}\bm{M}_{10}^{\frac{1}{2}}\:\hat{\bz}_{0}), \label{OMTrelaxMarg0}\\
    \hat{\sigma}_1(\hat{\bz}_{1}) &:= \det\left(\bm{M}_{10}\right)^{\frac{1}{2}} \sigma_{1}(\bm{M}_{10} ^{\frac{1}{2}}\:\hat{\bz}_{1}).\label{OMTrelaxMarg1}
\end{align}
\label{OMTrelaxMarg}	
\end{subequations}
In (\ref{OMTrelaxMarg}), $\sigma_{0},\sigma_{1}$ are as in (\ref{rho2sigma}). Furthermore, the matrices $\bm{M}_{10},\bm{\Phi}_{10}$ are the controllability Gramian and the state transition matrix, respectively, associated with $(\bA,\bb)$ in (\ref{feedbacklinearizedform}) over time horizon $[0,1]$. Exploiting the binary structure of the pair $(\bA,\bb)$, one can compute explicit formula for $\bm{M}_{10},\bm{\Phi}_{10}$ (see Appendix \ref{appendix:STMGramianFormula}), which will come in handy later.

Since (\ref{OMTrelax}) is a standard optimal transport problem, the optimal push-forward $\hat{T}^{\text{opt}}$ for this problem (i.e., $\hat{\sigma}_{1} = \hat{T}^{\text{opt}}\:\sharp\:\hat{\sigma}_{0}$) is unique \cite{villani2003topics}, and is of the form $\hat{T}^{\text{opt}}= \nabla\phi$, for $\phi$ convex. The function $\phi$ solves the Mong\'{e}-Amp{\`e}re PDE (see e.g., \cite[p. 282 and 323]{villani2008optimal})
\begin{align}
\det\left({\rm{Hess}}\left(\phi\right)\right) = \frac{\hat{\sigma}_{0}}{\hat{\sigma}_{1}\left(\nabla\phi\right)},
\label{MongeAmpere}	
\end{align}
which results from substituting $\hat{T}^{\text{opt}}= \nabla\phi$ in $\hat{\sigma}_{1} = \hat{T}^{\text{opt}}\:\sharp\:\hat{\sigma}_{0}$. The optimal coupling $\hat{\pi}^{\text{opt}}$ in (\ref{OMTrelax}) is supported on the graph of $\hat{T}^{\text{opt}}$. 
The minimizing pair $(\widetilde{\sigma}^{\text{opt}},\widetilde{v}^{\text{opt}})$ for (\ref{TransOCFluidsrelax}) can be computed by using $\hat{T}^{\text{opt}}$ as follows.

\begin{proposition}\label{PropOMTsemianalytical} 
Given the optimal transport map $\hat{T}^{{\rm{opt}}}=\nabla\phi$ for (\ref{OMTrelax}), where $\phi$ solves (\ref{MongeAmpere}), let
\begin{subequations}
\begin{align}
\bm{P}(t) &:= \bm{\Phi}(t,1)\bm{M}(1,t)\bm{M}_{10}^{-1}\bm{\Phi}_{10},\label{defP}\\
\bm{Q}(t) &:=\bm{M}(t,0)\bm{\Phi}(1,t)^{\top}\bm{M}_{10}^{-1},	
\end{align}
\label{DefPQ}	
\end{subequations}
wherein the controllability Gramian $\bm{M}$ and the state transition matrix $\bm{\Phi}$ are as in Appendix \ref{appendix:STMGramianFormula}. Define
\begin{subequations}
\begin{align}
T(\bz) &:= \bm{M}_{10}^{\frac{1}{2}}\hat{T}^{{\rm{opt}}}(\bm{M}_{10}^{-\frac{1}{2}}\bm{\Phi}_{10}\bz),\label{defT}\\
T_t(\bz) &:= \bm{P}(t)\bz + \bm{Q}(t)T(\bz).\label{DefTt}
\end{align}
\label{DefTTt}	
\end{subequations}
Then, the minimizing pair $(\widetilde{\sigma}^{{\rm{opt}}},\widetilde{v}^{{\rm{opt}}})$ for (\ref{TransOCFluidsrelax}) is given by
\begin{subequations}
\begin{align} 
\!\!\!\widetilde{\sigma}^{{\rm{opt}}}(\bz,t) \!&=  \left(T_t\right)\:\sharp\:\hat{\sigma}_{0}, \label{OMToptpdf}\\
\!\!\!\widetilde{v}^{{\rm{opt}}}(\bz,t) \!&= \bb^{\top}\bm{\Phi}(1,t)^{\top}\bm{M}^{-1}_{10}[T\circ T^{-1}_{t}(\bz) - \bm{\Phi}_{10}T^{-1}_{t}(\bz)] \label{OMToptvel}.
\end{align}
\label{OMTAnalyticalSol}	
\end{subequations}
\end{proposition}
\begin{proof}
See \cite[Section III.A]{chen2017optimal}; therein the authors also prove that $T_{t}$ in (\ref{DefTt}) is injective, i.e., the optimal pair (\ref{OMTAnalyticalSol}) is well-defined. We note here that \cite[eqn. (32)]{chen2017optimal} has a typo: a factor $\bm{\Phi}_{10}$ is missing after the minus sign in that expression; c.f. (\ref{OMToptvel}) above. 	
\end{proof}

\begin{figure*}[htpb]
\begin{subequations}
\begin{align}
T(\bz) &= \begin{pmatrix}
\sqrt{\frac{4}{3} + \frac{3\sqrt{3}}{52}} & \frac{3}{2}\bigg\{\sqrt{\frac{4+\sqrt{13}}{78}} - \sqrt{\frac{4-\sqrt{13}}{78}}\bigg\}\\
\frac{3}{2}\bigg\{\sqrt{\frac{4+\sqrt{13}}{78}} - \sqrt{\frac{4-\sqrt{13}}{78}}\bigg\} & \sqrt{\frac{10}{13} + \frac{3\sqrt{3}}{52}}	
\end{pmatrix}\nonumber\\
 &\times \hat{T}^{{\rm{opt}}}\left(\begin{pmatrix}
\sqrt{\frac{3(40+3\sqrt{3})}{13}} & \sqrt{\frac{3(40+3\sqrt{3})}{13}}+3\bigg\{\sqrt{\frac{4-\sqrt{13}}{26}} - \sqrt{\frac{4+\sqrt{13}}{26}}\bigg\}\\
3\bigg\{\sqrt{\frac{4-\sqrt{13}}{26}} - \sqrt{\frac{4+\sqrt{13}}{26}}\bigg\}	& \sqrt{\frac{16+9\sqrt{3}}{13}}+3\bigg\{\sqrt{\frac{4-\sqrt{13}}{26}} - \sqrt{\frac{4+\sqrt{13}}{26}}\bigg\} 
\end{pmatrix}
\bz\right),\\
T_t(\bz) &= \begin{pmatrix}
 -2(1-t)^{3} + 3(1-t)^{2} & -(1-t)^{3}+(1-t)^{2}\\
 6(1-t)^{2}-6(1-t) & 3(1-t)^{2}-2(1-t)	
 \end{pmatrix}
\bz + \begin{pmatrix}
 	-2t^{3} + 3t^{2} & t^{3}-t^{2}\\
 -6t^{2}+6t & 3t^{2}-2t
 \end{pmatrix}
T(\bz).	
\end{align}
\label{TTtDI}
\end{subequations}
\end{figure*}

While Proposition \ref{PropOMTsemianalytical} provides a semi-analytical handle for $(\widetilde{\sigma}^{\text{opt}},\widetilde{v}^{\text{opt}})$, the associated computation is challenging since obtaining $\hat{T}^{{\rm{opt}}}$ requires $\phi$, which in turn requires solving the second order elliptic PDE (\ref{MongeAmpere}). Numerical methods for solving the latter is a topic of recent research \cite{froese2011convergent,benamou2014numerical}. Even if one has access to $\hat{T}^{{\rm{opt}}}$, the remaining analytical computation in Proposition \ref{PropOMTsemianalytical} is rather unwieldy. We illustrate this through the double integrator example below. After that, we outline an alternative approach for solving (\ref{TransOCFluidsrelax}) in Section \ref{SubsecHJBapprox}.

\begin{example}(\textbf{Double integrator})
If the state space dimension $n=2$ in (\ref{SingleInputControlAffine}), then (\ref{feedbacklinearizedform}) reduces to a double integrator. Suppose that one has already computed the map $\hat{T}^{{\rm{opt}}}=\nabla\phi$ as function of the given endpoint PDFs $\rho_{0},\rho_{1}$ via (\ref{rho2sigma}), (\ref{OMTrelaxMarg}), and (\ref{MongeAmpere}). Using Appendix \ref{appendix:STMGramianFormula}, (\ref{DefPQ}) and (\ref{DefTTt}), we then get (\ref{TTtDI}). In this case, (\ref{OMToptvel}) results
\begin{align*}
&\widetilde{v}^{{\rm{opt}}}(\bz,t) \\
&= \begin{pmatrix}
 	6-12t & -2+6t
 \end{pmatrix}
\left[T\circ T^{-1}_{t}(\bz) - \begin{pmatrix}
 1 & 1\\
 0 & 1	
 \end{pmatrix}
T^{-1}_{t}(\bz)\right],
\end{align*}
where $T(\cdot),T_{t}(\cdot)$ are given by (\ref{TTtDI}).
\end{example}

\subsection{Solution of (\ref{TransOCFluidsrelax}) via HJB PDE}\label{SubsecHJBapprox}
The following Proposition summarizes an alternative approach for solving (\ref{TransOCFluidsrelax}) compared to the one in Section \ref{SubsecOMTapprox}.
\begin{proposition}\label{PropHJBapprox}
Let $\widetilde{\psi}(\bz,t)$ solves the HJB PDE
\begin{equation}
\begin{aligned}
\!\!\!\frac{\partial\widetilde{\psi}}{\partial t} + \!\displaystyle\sum_{i=1}^{n-1}\!z_{i+1}\frac{\partial\widetilde{\psi}}{\partial z_{i}} + \frac{1}{2}\!\left(\!\frac{\partial\widetilde{\psi}}{\partial z_{n}}\!\right)^{\!\!2}\!\!=\! 0,
\end{aligned}
\label{HJBapprox}
\end{equation}
subject to terminal condition $\widetilde{\psi}(\bz,1) = \widetilde{\psi}_{1}(\bz)$ that depends on $\hat{\sigma}_{0},\hat{\sigma}_{1}$. Then, $\widetilde{v}^{{\rm{opt}}}$ in (\ref{TransOCFluidsrelax}) can be obtained as 
\begin{equation}
    \widetilde{v}^{{\rm{opt}}}(\bz,t) = \displaystyle\frac{\partial}{\partial z_{n}}\widetilde{\psi}(\bz,t).
    \label{HJBapproxv} 
\end{equation}
The optimal controlled PDF $\widetilde{\sigma}^{{\rm{opt}}}(\bz,t)$ solves (\ref{ConstrLiouvillezRelax})-(\ref{ConstrPDFzRelax}) with $\widetilde{v}\equiv\widetilde{v}^{{\rm{opt}}}(\bz,t)$ given by (\ref{HJBapproxv}).
\end{proposition}
\begin{proof}
Follows from specializing Theorem \ref{ThmHJBopt} for $\alpha_{\bm{\tau}}(\bz)\equiv 0$, $\beta_{\bm{\tau}}(\bz)\equiv 1$ for all $\bz\in\mathcal{Z}$; cf. (\ref{TransformedLOMT}) with (\ref{TransOCFluidsrelax}).	
\end{proof}

\subsubsection{Envelope representation for $\widetilde{\psi}$}\label{subsubEnvApprox}  
Following \cite{rockafellar2000convexity}, we next give upper and lower envelope representation formula for $\widetilde{\psi}(\bz,t)$ in Proposition \ref{PropHJBapprox} via Fenchel duality; see Appendix \ref{appendix:EnvelopeFormula}. The upper envelope representation formula (\ref{UpperEnvelope}) below has appeared in \cite{chen2017optimal} but followed a  different derivation (see \cite[Appendix A]{chen2017optimal}) compared to our duality approach. The lower envelope representation formula (\ref{LowerEnvelope}) below is new, and comes out of the same duality framework.
  
\begin{proposition}\label{PropEnvelopeApprox}
Consider $(\bA,\bb)$ as in (\ref{feedbacklinearizedform}) and the associated controllability Gramian $\bm{M}(t,s)$ given in Appendix \ref{appendix:STMGramianFormula}. For any initial function $\widetilde{\psi}_{0}(\bz):=\widetilde{\psi}(\bz,t=0):\mathcal{Z}\mapsto\mathbb{R}\cup\{+\infty\}$, the map $\widetilde{\psi}(\bz,t)$ in Proposition \ref{PropHJBapprox} admits the following upper envelope representation:
\begin{align}
\widetilde{\psi}(\bz,t) \!= \underset{\overline{\bz}}{\inf}\bigg\{\!\widetilde{\psi}_{0}(\overline{\bz}) + \frac{1}{2}\left\lVert \bz - \exp(t\bA)\overline{\bz} \right\rVert_{\left(\bM(t,0)\right)^{\!-1}}^{2}\!\!\bigg\}.
\label{UpperEnvelope}	
\end{align}
If $\widetilde{\psi}_{0}(\bz)$ is convex, proper and lower semicontinuous in $\bz\in\mathcal{Z}$, then 	the map $\widetilde{\psi}(\bz,t)$ in Proposition \ref{PropHJBapprox} also admits the following lower envelope representation:
\begin{align}
&\widetilde{\psi}(\bz,t) = \underset{\bm{r}}{\sup}\bigg\{-\widetilde{\psi}_{0}^{*}(\bm{r}) + \bz^{\top}\exp\left(-t\bA^{\top}\right)\bm{r} \nonumber\\
&- \frac{1}{2}\bm{r}^{\top}\!\left(\!\int_{0}^{t}\!\exp(-s\bA)\bb\bb^{\top}\exp(-s\bA^{\top})\differential s\!\right)\!\bm{r}\bigg\}.
\label{LowerEnvelope}
\end{align}
\end{proposition}

\begin{proof}
In view of Appendix \ref{appendix:EnvelopeFormula}, notice that (\ref{HJBapprox}) is of the form (\ref{HJBform}) with Hamiltonian $\mathcal{H}(\bz,\bm{\zeta}) = \langle \bA\bz, \bm{\zeta} \rangle + \mathcal{H}_{0}(\bm{\zeta})$, where $\mathcal{H}_{0}(\bm{\zeta}) = (\bb^{\top}\bm{\zeta})^{2}/2$ is convex in $\bm{\zeta}$. Thus, $\mathcal{H}$ is concave in $\bz$, and convex in $\bm{\zeta}$. In this case, it is easy to verify that $\mathcal{H}$ satisfies the conditions (H1)-(H2) of Theorem \ref{RWthm} in Appendix \ref{appendix:EnvelopeFormula} (see  \cite[p. 470]{goebel2002generalized}). Therefore, the envelope formulas (\ref{GeneralUpperEnvelope})-(\ref{GeneralLowerEnvelope}) apply.

Following \cite[Example 5.2]{rockafellar2000convexity}, (\ref{GeneralUpperEnvelope}) becomes
\begin{align}
&\widetilde{\psi}(\bz,t) = \underset{\overline{\bz}}{\inf}\bigg\{ \widetilde{\psi}_{0}(\overline{\bz})\nonumber\\
 &+ \left(\int_{0}^{t}\!\!\!\mathcal{H}_{0}\left(\exp\left(-s\bA^{\top}\right)\bm{\zeta}\right)\:\differential s\right)^{*}\!\!(\cdot)\Bigg\vert_{(\cdot) = \exp(-t\bA)\bz - \overline{\bz}} \bigg\}.
\label{5point11inRockafellar}	
\end{align}
Direct computation gives 
\begin{align*}
&\left(\int_{0}^{t}\mathcal{H}_{0}\left(\exp\left(-s\bA^{\top}\right)\bm{\zeta}\right)\differential s\right)^{*}\left(\cdot\right) \\
&=\frac{1}{2}\left(\cdot\right)^{\top}\!\bigg\{\!\int_{0}^{t}\!\!\exp(-s\bA)\bb\bb^{\top}\!\exp(-s\bA^{\top})\differential s\!\bigg\}^{-1}\left(\cdot\right),
\end{align*}
and therefore, the second summand in (\ref{5point11inRockafellar}) equals
\begin{align}
\!\frac{1}{2}\!\left(\bz - \exp(t\bA)\overline{\bz}\right)^{\!\top}\!\!\bigg\{\!\!\int_{0}^{t}\!\!&\exp((t-s)\bA)\bb\bb^{\top}\!\!\exp((t-s)\bA^{\top})\differential s\!\bigg\}^{\!\!-1}\nonumber\\
&\times\left(\bz - \exp(t\bA)\overline{\bz}\right),
\label{Simplify2ndSummand}
\end{align}
which is precisely $\frac{1}{2}\lVert \bz - \exp(t\bA)\overline{\bz} \rVert_{(\bm{M}(t,0))^{-1}}^{2}$, as claimed in (\ref{UpperEnvelope}). Likewise, specializing (\ref{LowerEnvelopeConjugateform}) for $\mathcal{H}(\bz,\bm{\zeta}) = \langle\bA\bz,\bm{\zeta}\rangle + (\bb^{\top}\bm{\zeta})^{2}/2$ yields (\ref{LowerEnvelope}).
\end{proof}

Both the terminal condition $\widetilde{\psi}_{1}(\bz)$ in Proposition \ref{PropHJBapprox}, and the initial condition $\widetilde{\psi}_{0}(\bz)$ in Proposition \ref{PropEnvelopeApprox} can be written in terms of the endpoint PDFs $\sigma_{0},\sigma_{1}$ in  
(\ref{ConstrPDFzRelax}) as follows.

\begin{proposition}\label{PropTerminalInitial}
Given endpoint PDFs $\sigma_{0},\sigma_{1}$ in  
(\ref{ConstrPDFzRelax}), construct the pair $(\hat{\sigma}_{0},\hat{\sigma}_{1})$ using (\ref{OMTrelaxMarg}), and let $\phi$ be the solution of (\ref{MongeAmpere}). Let the controllability Gramian $\bM$ and the state transition matrix $\bm{\Phi}$ associated with $(\bA,\bb)$ in (\ref{feedbacklinearizedform}), be given by Appendix \ref{appendix:STMGramianFormula}. Then, $\widetilde{\psi}_{1}(\bz)$ in Proposition \ref{PropHJBapprox}, and $\widetilde{\psi}_{0}(\bz)$ in Proposition \ref{PropEnvelopeApprox}, are respectively given by
\begin{align}
&\widetilde{\psi}_{1}(\bz) = \frac{1}{2}\bz^{\top}\bM_{10}^{-1}\bz - \phi^{*}\left(\bm{M}_{10}^{-1/2}\bz\right),
\label{psioneviaphi}\\	
&\widetilde{\psi}_{0}(\bz) = \phi\left(\bM_{10}^{-1/2}\bm{\Phi}_{10}\bz\right) - \frac{1}{2}\bz^{\top}\bm{\Phi}_{10}^{\top}\bM_{10}^{-1}\bm{\Phi}_{10}\bm{z}.
\label{psizeroviaphi}	
\end{align}	
\end{proposition}
\begin{proof}
To derive (\ref{psioneviaphi}), first note from (\ref{HJBapproxv}) that 
\begin{align}
\widetilde{v}^{\rm{opt}}(\bz,1)=\bb^{\top}\nabla_{\bz}\widetilde{\psi}_{1}.
\label{psi1fromHJB}	
\end{align}
Next, we substitute $t=1$ in (\ref{OMToptvel}), and notice from (\ref{DefPQ}) that $\bm{P}(1)=\bm{0}$, $\bm{Q}(1)=\bm{I}$, which imply $T_{1}(\cdot) = T(\cdot)$. Consequently, the right-hand-side of (\ref{OMToptvel}) at $t=1$ becomes 
\begin{align}
\bb^{\top}\bM_{10}^{-1}\left[\bz - \bm{\Phi}_{10}T^{-1}(\bz)\right].
\label{RHSfort1}	
\end{align}
From (\ref{defT}), we have that
\begin{align}
T^{-1}(\cdot) = \bm{\Phi}_{10}^{-1}\bM_{10}^{1/2}\left(\hat{T}^{\rm{opt}}\right)^{-1}\left(\bM_{10}^{-1/2}\cdot\right),
\label{Tinvfort1}	
\end{align}
which together with (\ref{RHSfort1}) and (\ref{psi1fromHJB}) gives
\begin{align}
\nabla_{\bz}\widetilde{\psi}_{1} = \bM_{10}^{-1}\left[\bz - \bm{M}_{10}^{1/2}\left(\nabla_{\bz}\phi\right)^{-1}\left(\bM_{10}^{-1/2}\bz\right)\right].	
\label{Intermedpsi1}
\end{align}
We clarify here that the notation $\left(\nabla_{\bz}\phi\right)^{-1}\left(\bM_{10}^{-1/2}\bz\right)$ reads inverse function of $\nabla_{\bz}\phi$ evaluated at $\bM_{10}^{-1/2}\bz$.

Denote the Legendre-Fenchel conjugate of $\phi(\bz)$ as $\phi^{*}(\bz^{*})$, and recall the well-known relation $\left(\nabla_{\bz}\phi\right)^{-1} = \nabla_{\bz^{*}}\phi^{*}$. Thus, (\ref{Intermedpsi1}) yields (\ref{psioneviaphi}). 

To derive (\ref{psizeroviaphi}), we substitute $t=0$ in (\ref{OMToptvel}), and notice from (\ref{DefPQ}) that $\bm{P}(0)=\bm{I}$, $\bm{Q}(0)=\bm{0}$, which imply $T_{0},T_{0}^{-1}$ are identity maps. Consequently, the right-hand-side of (\ref{OMToptvel}) at $t=0$ becomes 
\begin{align}
\bb^{\top}\bm{\Phi}_{10}^{\top}\bm{M}_{10}^{-1}\left[\bm{M}_{10}^{1/2}\nabla_{\bz}\phi\left(\bM_{10}^{-1/2}\bm{\Phi}_{10}\bz\right) - \bm{\Phi}_{10}\bz\right].
\label{RHSapproxcontrol}	
\end{align}
On the other hand, from (\ref{HJBapproxv}) we have $\widetilde{v}^{\rm{opt}}(\bz,0)=\bb^{\top}\nabla_{\bz}\widetilde{\psi}_{0}$. Equating these, we obtain (\ref{psizeroviaphi}).
\end{proof}

\subsubsection{Characteristic representation for $\widetilde{\psi}$}\label{subsubCharApprox} 
As in Section \ref{subsubChar}, an alternative representation of $\widetilde{\psi}(\bz,t)$ in Proposition \ref{PropHJBapprox} can be obtained via characteristic ODEs for the first order PDE (\ref{HJBapprox}). In particular, specializing Theorem \ref{ThmJHBoptMOC} to the case $\alpha_{\bm{\tau}}(\bz)\equiv 0$, $\beta_{\bm{\tau}}(\bz)\equiv 1$, the characteristic equations reduce to 
\begin{equation}
	\begin{aligned}
&\frac{\differential t}{\differential s}=1,	\quad \frac{\differential \bm{z}}{\differential s}=\bA \bm{z} + \bb\bb^{\top} \bm{\zeta}, \quad \frac{\differential \bm{\zeta}}{\differential s}=-\bA^{\top} \bm{\zeta},\\
&\frac{\differential \theta} {\differential s}=0, \quad \frac{\differential \widetilde{\psi}} {\differential s}=(\bb^{\top} \bm{\zeta})^2/2,
\end{aligned}
\label{CharODEapprox}
\end{equation}
where the tuple $\left(t(s),\bz(s),\bm{\zeta}(s),\theta(s),\widetilde{\psi}(s)\right)$ parametrized via $s$, denote the characteristic strip  associated with (\ref{HJBapprox}). We next solve the system of ODEs (\ref{CharODEapprox}) to furnish the solution of (\ref{HJBapprox}) as function of the initial data $\widetilde{\psi}_{0}$, and independent variables $\bz,t$.

From (\ref{CharODEapprox}), notice that $s=t$, and let the initial condition $(\bz(0),\bm{\zeta}(0),\theta(0),\widetilde{\psi}(0)) \equiv (\bz_{0},\bm{\zeta}_{0},\theta_{0},\widetilde{\psi}_{0})$, where $\widetilde{\psi}_{0}(\bz_{0})$ is given by (\ref{psizeroviaphi}). Since $\nabla_{\bz_{0}}\widetilde{\psi}	_{0} = \langle\nabla_{\bz}\widetilde{\psi}_{0}, \nabla_{\bz_{0}}\bz\rangle\big\rvert_{s=0} = \bm{\zeta}_{0}$, therefore using (\ref{psizeroviaphi}), we can determine $\bm{\zeta}_{0}(\bz_{0})$ as 
\begin{align}
\bm{\zeta}_{0} = \exp(\bA^{\top})\bM_{10}^{-1/2}&\nabla\phi\left(\bM_{10}^{-1/2}\exp(\bA)\bz_{0}\right) \nonumber\\
&- \exp(\bA^{\top})\bM_{10}^{-1}\exp(\bA)\bm{z}_{0}. 
\label{zeta0asfuncofz0}
\end{align}
Furthermore, from (\ref{CharODEapprox}), we have
\begin{subequations}
\begin{align} 
\bm{\zeta}(t) &= \exp(-t\bA^{\top})\bm{\zeta}_0,\label{zetaapprox}\\
\bz(t) &= \exp(t\bA)\!\left[\bz_{0} \!+\! \bigg\{\!\!\int_{0}^{t}\!\!\exp(-\tau\bA)\bb\bb^{\!\top}\!\exp(-\tau\bA^{\!\top}\!)\differential\tau\!\!\bigg\}\bm{\zeta}_{0}\right]\nonumber\\
&= \exp(t\bA)\bz_{0} + \bm{M}(t,0)\exp(-t\bA^{\top})\bm{\zeta}_{0},\label{zapprox}	
\end{align}
\label{CharApproxSoln}	
\end{subequations}
where the last step follows by introducing a change-of-variable $\nu := t-\tau$ in the integral within curly braces, with lower limit $\nu=t$ and upper limit $\nu=0$. 

Combining (\ref{zeta0asfuncofz0}) and (\ref{CharApproxSoln}), we can express $\bm{\zeta},\bz$ as function of $t,\bz_{0}$, i.e., 
\begin{subequations}
\begin{align}
&\bm{\zeta} = \exp\left((1-t)\bA^{\top}\right)\bM_{10}^{-1/2}\nabla\phi\left(\bM_{10}^{-1/2}\exp(\bA)\bz_{0}\right) \nonumber\\
& \qquad\qquad\qquad - \exp((1-t)\bA^{\top})\bM_{10}^{-1}\exp(\bA)\bm{z}_{0}, \label{zetaCharSoln}	\\
&\bz = \exp(t\bA)\bz_{0} + \bM(t,0)\bm{\zeta}.\label{zzetaz0}
\end{align}
\label{zzetaasfuncoftz0}	
\end{subequations}
Notice that the last ODE in (\ref{CharODEapprox}) is 
\[\frac{\differential\widetilde{\psi}}{\differential t} = \frac{1}{2}\bm{\zeta}^{\top}\bb\bb^{\top}\bm{\zeta},\]
which combined with (\ref{zetaapprox}), followed by integration results
\begin{align}
&\widetilde{\psi}(\bz,t) - \widetilde{\psi}_{0}(\bz) = \frac{1}{2}\bm{\zeta}^{\top}\bM(t,0)\bm{\zeta} \nonumber\\
\Rightarrow &\widetilde{\psi}(\bz,t) =\widetilde{\psi}_{0}(\bz) + \frac{1}{2}\Vert \bz - \bm{\Phi}(t,0)\bz_{0}(\bz,t)\Vert_{\left(\bM(t,0)\right)^{-1}}^{2},
\label{psiviaz0t}	
\end{align}
where the last line follows from (\ref{zzetaz0}). 

Let the nonsingular matrix $\bm{R}:=\bM_{10}^{-1/2}\exp(\bA)$. By eliminating $\bm{\zeta}$ from (\ref{zetaCharSoln})-(\ref{zzetaz0}), the map $\bz_{0}(\bz,t)$ in (\ref{psiviaz0t}) can be obtained as the solution of the nonlinear equation
\begin{eqnarray}
\exp(t\bA)\bz_{0} + \bM(t,0)\exp(-t\bA^{\top})\left[\bm{R}^{\top}\hat{T}^{{\rm{opt}}}\left(\bm{R}\bz_{0}\right) \right.\nonumber\\
 \left.- \bm{R}^{\top}\bm{R}\bz_{0}\right] = \bz.
\label{z0viaz}	
\end{eqnarray}
The existence and uniqueness of solution (\ref{psiviaz0t})-(\ref{z0viaz}) follows from Appendix \ref{appendix:MOCinvfnthm}. Therefore, (\ref{psiviaz0t})-(\ref{z0viaz}) together furnish the solution of (\ref{HJBapprox}). 


\section{PDF Steering via Schr{\"o}dinger Bridge}\label{SecSchrodinger}
We now outline a computational framework for solving (\ref{TransOCFluidsrelax}) via the stochastic control formulation of the classical Schr{\"o}dinger bridge \cite{schrodinger1931umkehrung,schrodinger1932theorie}. In its original formulation \cite{schrodinger1931umkehrung,schrodinger1932theorie}, Schr{\"o}dinger bridge determines a probability measure $\mathbb{P}$ on $\Omega:=C([0,1],\mathcal{X})$, i.e., on continuous paths from $t=0$ to $t=1$ in the state space $\mathcal{X}\subseteq\mathbb{R}^{n}$, that minimizes the relative entropy $\int_{\Omega}\log(\differential\mathbb{P}/\differential\mathbb{Q}) \differential\mathbb{P}$ where $\mathbb{Q}$ is the probability measure induced by a prescribed prior Markovian dynamics, and $\differential\mathbb{P}/\differential\mathbb{Q}$ denotes the Radon-Nikodym derivative. The measure $\mathbb{P}$ is assumed to be absolutely continuous w.r.t. $\mathbb{Q}$, and is constrained to admit endpoint marginal measures $\rho_{0}(\bx)\differential\bx$ at $t=0$, and $\rho_{1}(\bx)\differential\bx$ at $t=1$, with given $\rho_{0},\rho_{1}$ satisfying assumption \textbf{(A1)} in Section \ref{SubsecReformulation}. We refer the readers to \cite{bernstein1932liaisons,beurling1960automorphism,jamison1974reciprocal,wakolbinger1990schrodinger} for representative references; see \cite{leonard2014survey} for a recent survey. 

Here we focus on the stochastic control formulation \cite{dai1991stochastic,blaquiere1992controllability} of the Schr{\"o}dinger bridge viewed as a dynamic stochastic regularization of (\ref{TransOCFluidsrelax}), i.e., we consider the problem
\begin{subequations}
\begin{align}
&\qquad\;\underset{(\widetilde{\sigma},\widetilde{v})}{\text{inf}}
& & \int_{\mathcal{Z}} \int_{0}^{1} \frac{1} {2}   \lvert \widetilde{v}(\bz,t) \rvert^{2} \,\widetilde{\sigma}(\bz,t) \: \differential t \: \differential\bz\label{ObjSchrodingerRelax}\\ 
& \text{subject to} & &  \dfrac{\partial \widetilde{\sigma}}{\partial t} + \nabla_{\bz} \cdot ((\bA\bz + \bb \widetilde{v} )\widetilde{\sigma}) = \epsilon\Delta\widetilde{\sigma},\label{SchrodingerConstrFPKRelax}\\
& & &  \widetilde{\sigma}(\bz,0)=\sigma_{0}(\bz),  \quad \widetilde{\sigma}(\bz,1)=\sigma_{1}(\bz).\label{SchrodingerConstrPDFzRelax}
\end{align}
\label{RelaxedSchrodinger} 
\end{subequations}
In above, the difference with (\ref{TransOCFluidsrelax}) is that the controlled Liouville PDE (\ref{ConstrLiouvillezRelax}) is replaced by the Fokker-Planck-Kolmogorov PDE (\ref{SchrodingerConstrFPKRelax}) with same drift as before, and diffusion coefficient $\sqrt{2\epsilon}$, for some regularization parameter $\epsilon > 0$. Put differently, replacing (\ref{ConstrLiouvillezRelax}) with (\ref{SchrodingerConstrFPKRelax}) is equivalent to regularizing the controlled sample path ODE 
\begin{align}
\dot{\bz}(t) = \bA\bz(t) + \bb\:\widetilde{v}(\bz,t),
\label{ControlledSamplePathODE}	
\end{align}
with the It\^{o} SDE
\begin{align}
\differential\bz(t) = \left(\bA\bz(t) + \bb\:\widetilde{v}(\bz,t)\right)\:\differential t + \sqrt{2\epsilon}\:\bb\:\differential w(t),
\label{ControlledSamplePathSDE}	
\end{align}
where $w(t)$ denotes standard Wiener process. As $\epsilon\downarrow 0$, the solution of (\ref{RelaxedSchrodinger}), which we denote by $(\widetilde{\sigma}_{\epsilon}^{\rm{opt}},\widetilde{v}_{\epsilon}^{\rm{opt}})$, is known \cite{mikami2004monge,leonard2012schrodinger,chen2016relation,chen2017optimal} to converge to $(\widetilde{\sigma}^{\rm{opt}},\widetilde{v}^{\rm{opt}})$, i.e., to the solution of (\ref{TransOCFluidsrelax}). This suggests numerically solving (\ref{RelaxedSchrodinger}) for small $\epsilon$ to approximate the solution of (\ref{TransOCFluidsrelax}). The idea is appealing since the solution of (\ref{RelaxedSchrodinger}) is known to be
\begin{subequations}
\begin{align}
&\widetilde{\sigma}_{\epsilon}^{\rm{opt}}(\bz,t) = \hat{h}(\bz,t)h(\bz,t), \quad 0\leq t\leq1,\label{SBPsolnPDF}\\
& \widetilde{v}^{\rm{opt}}_{\epsilon}(\bz,t) = 2\epsilon\:\bb^{\top}\nabla_{\bz} h(\bz,t), \quad 0\leq t\leq1,\label{SBPsolnControl}\\
&\hat{h}\left(\bz,t\right) = \int_{\mathcal{Z}}\kappa\left(0,\overline{\bz},t,\bz\right) \hat{h}_{0}(\overline{\bz})\differential\overline{\bz}, \quad t\geq0, \label{SBPsolnForwardKernel}\\
& h\left(\bz,t\right) = \int_{\mathcal{Z}}\kappa\left(t,\bz,1,\overline{\bz}\right) h_{1}(\overline{\bz})\differential\overline{\bz},	\quad t\leq1,\label{SBPsolnBackwardKernel}
\end{align}
\label{SBPsystem}	
\end{subequations}
where $\kappa\left(s,\overline{\bz},t,\bz\right)$ is the Markov kernel associated with (\ref{SchrodingerConstrFPKRelax}) that depends on $\epsilon$, and the factors in (\ref{SBPsolnPDF}) have boundary values $\hat{h}_{0}(\bz):=\hat{h}(\bz,0) \geq 0$ and $h_{1}(\bz):=h(\bz,1)\geq 0$. Combining (\ref{SBPsystem}) with the boundary conditions (\ref{ConstrPDFzRelax}) yield the following system of nonlinear integral equations
\begin{subequations}
\begin{align}
\hat{h}_{0}(\bz)\int_{\mathcal{Z}}\kappa\left(0,\bz,1,\overline{\bz}\right) h_{1}(\overline{\bz})\differential\overline{\bz} &= \sigma_{0}(\bz),\label{IE0}\\
h_{1}(\bz)\int_{\mathcal{Z}}\kappa\left(0,\overline{\bz},1,\bz\right) \hat{h}_{0}(\overline{\bz})\differential\overline{\bz} &= \sigma_{1}(\bz),\label{IE1}		
\end{align}
\label{IntegralEquations}	
\end{subequations}
which can be solved for the pair $(\hat{h}_{0},h_{1})$ as fixed point recursion with guaranteed convergence properties; see \cite{chen2016entropic}. The converged pair $(\hat{h}_{0},h_{1})$ can then be used in (\ref{SBPsolnForwardKernel})-(\ref{SBPsolnBackwardKernel}) to determine the pair $(\hat{h}(\bz,t),h(\bz,t))$, and thereby determine the pair $(\widetilde{\sigma}_{\epsilon}^{\rm{opt}}(\bz,t),\widetilde{v}^{\rm{opt}}_{\epsilon}(\bz,t))$ via (\ref{SBPsolnPDF})-(\ref{SBPsolnControl}). We remind the readers that since $\kappa$ depends on $\epsilon$, so does the pair $(\hat{h}_{0},h_{1})$, and thus the pair $(\hat{h}(\bz,t),h(\bz,t))$ too depends on $\epsilon$. We will pursue this approach to solve (\ref{RelaxedSchrodinger}) in our numerical example in Section \ref{SecNumExample}. In contrast with the Schr\"{o}dinger bridge regularization, numerical solution of dynamic optimal transport formulation such as (\ref{TransOCFluidsrelax}) remains challenging in high dimensional state space -- some specialized algorithms \cite{benamou2000computational,angenent2003minimizing} for the same have been proposed in the literature.

\begin{algorithm}
    \caption{Schr\"{o}dinger Bridge Regularization (SBR) to compute  $\rho_{\epsilon}(\bx,t)$}
    \label{AlgoSBR}
    \begin{algorithmic}[1] 
        \Procedure{SBR}{$\bm{\rho}_0$, $\bm{\rho}_1, n,{\rm{nIter}},\bm{\tau},\kappa^{\rm{B}},\epsilon,\delta$}

            \State $\bA \gets {\rm{diag}}({\rm{ones}}(n-1,1),1)$
            \State $\bm{b} \gets [{\rm{zeros}}(1,n-1),1]^{\top} $ 
            
            \State $\bm{\sigma}_0 \gets \bm{\tau}\sharp\bm{\rho}_0$ \Comment{pushforward $\rho_{0}$}
             \State $\bm{\sigma}_1 \gets \bm{\tau}\sharp\bm{\rho}_1$ \Comment{pushforward $\rho_{1}$}
            \State $\bm{K}_{st} \gets {\rm{LTIkernel}}(\bA,\bm{b},\kappa^{\rm{B}}(\epsilon,s,t),s,t)$ \Comment{kernel (64)}
            
            \State $\hat{\bm{h}}_1 \gets \left[\bm{1}_{n\times 1}, \bm{0}_{n\times ({\rm{nIter}}-1)}\right]$ \Comment{initialize}
            \State $\bm{h}_1 \gets \left[ \bm{0}_{n \times {\rm{nIter}}}\right] $
            \State $\bm{h}_0 \gets \left[ \bm{0}_{n \times {\rm{nIter}}}\right] $
            \State $\hat{\bm{h}}_0 \gets \left[ \bm{0}_{n \times {\rm{nIter}}}\right] $ 
            \State $\ell=1$ \Comment{iteration index}
            
            \While{$\ell \leq {\rm{nIter}} $} \Comment{fixed point iteration}
                \State $\bm{h}_1(:,\ell +1) \gets  \bm{\rho}_1 \oslash \hat{\bm{h}}_1(:,\ell) $
                \State $\bm{h}_0(:,\ell+1) \gets \bm{K}_{01}\bm{h}_1(:,\ell +1)   $
                \State $\hat{\bm{h}}_0(:,\ell+1) \gets \bm{\rho}_0 \oslash  \bm{h}_0(:,\ell+1) $
                \State $\hat{\bm{h}}_1(:,\ell+1) \gets \bm{K}^{\top}_{01}\bm{h}_1(:,\ell +1)  $
                \If{$\parallel \hat{\bm{h}}_0(:,\ell+1) - \hat{\bm{h}}_0(:,\ell)\parallel < \delta \;\And\; \parallel \bm{h}_1(:,\ell+1)- \bm{h}_1(:,\ell)\parallel < \delta$ } \Comment{error within tolerance}
                \State break
                \Else
                \State $\ell \gets \ell + 1$
                \EndIf
            \EndWhile\label{euclidendwhile}
            \State \textbf{return} $\hat{\bm{h}}_0(:,\ell+1),\bm{h}_1(:,\ell+1)$ \Comment{converged pair}
            \State $\hat{\bm{h}}_t \gets K^{\top}_{0t} \hat{\bm{h}}_0(:,\ell+1)$
            \State $\bm{h}_t \gets K_{t1}\bm{h}_1(:,\ell+1)$  \Comment{transient Schr\"{o}dinger factors}
            \State $\bm{\sigma}_{\epsilon,t} \gets \bm{h}_t \odot \bm{\hat{h}}_t $ 
            \State $\bm{\rho}_{\epsilon,t} \gets \bm{\sigma}_{\epsilon,t}/ \lvert {\rm{det}} \left(\nabla \bm{\tau}\right) \rvert  $ \Comment{from (16a)}
            
            \EndProcedure
    \end{algorithmic}
\end{algorithm}

In passing, we mention that (\ref{IntegralEquations}) is not the classical Schr\"{o}dinger system in the sense that the Markov kernel $\kappa$ therein is \emph{not} the Brownian kernel $\kappa^{{\rm{B}}}$ associated with the $n$-dimensional scaled Wiener process $\sqrt{2\epsilon}\:\differential \bm{w}(t)$, given by
\begin{align}
\kappa^{{\rm{B}}}\left(s,\overline{\bz},t,\bz\right) = \left(4\pi(t-s)\epsilon\right)^{-n/2}\exp\left(\!-\frac{\Vert\overline{\bz}-\bz\Vert_{2}^{2}}{4\epsilon(t-s)}\!\right).
\label{BrownianKernel}	
\end{align}
But the two can be related through the formula \cite[Appendix B]{chen2017optimal}
\begin{align}
&\kappa\left(s,\overline{\bz},t,\bz\right) = \left(t-s\right)^{n/2} \det\left(\bM(t,s)\right)^{-1/2} \nonumber\\
&\times \kappa^{{\rm{B}}}\left(s,\left(\bM(t,s)\right)^{-1/2}\bm{\Phi}(t,s)\overline{\bz},t,\left(\bM(t,s)\right)^{-1/2}\bz\right),
\label{kappaViaBrownianKernel}	
\end{align}
where the matrices $\bm{\Phi},\bM$ are as in Appendix \ref{appendix:STMGramianFormula}. The formula (\ref{kappaViaBrownianKernel}) allows us to perform the fixed point recursion in a standard Schr\"{o}dinger system where the $\kappa$ in (\ref{IntegralEquations}) is replaced by $\kappa^{{\rm{B}}}$ given by (\ref{BrownianKernel}), and the PDF pair $\left(\sigma_{0},\sigma_{1}\right)$ in (\ref{IntegralEquations}) is replaced by $\left(\hat{\sigma}_{0},\hat{\sigma}_{1}\right)$ given by (\ref{OMTrelaxMarg}). Let the solution of the resulting classical Schr\"{o}dinger system be $(\hat{h}_{0}^{{\rm{B}}},h_{1}^{{\rm{B}}})$. Then the pair $(\hat{h}_{0},h_{1})$ in (\ref{IntegralEquations}) can be recovered\footnote{Recall that $\vert\det\left(\bm{\Phi}_{10}\right)\vert = 1$ in our case.} as \cite[eqn. (63)]{chen2017optimal}
\begin{subequations}
\begin{align}
\hat{h}_{0}(\bz) &= \hat{h}_{0}^{{\rm{B}}}\left(\bM_{10}^{-1/2}\bm{\Phi}_{10}\bz\right), \\
h_{1}(\bz) &= \det\left(\bM_{10}\right)^{-1/2}h_{1}^{{\rm{B}}}\left(\bM_{10}^{-1/2}\bz\right).  	
\end{align}
\label{RecoverLinSchrodingerSystemSoln}	
\end{subequations}

We summarize the computational pipeline in Algorithm \ref{AlgoSBR} and employ the same for the numerical examples presented next in Section \ref{SecNumExample}.

\begin{figure*}[t]
\centering
\begin{subfigure}{0.5\textwidth}
    \centering
        \includegraphics[width=0.98\linewidth]{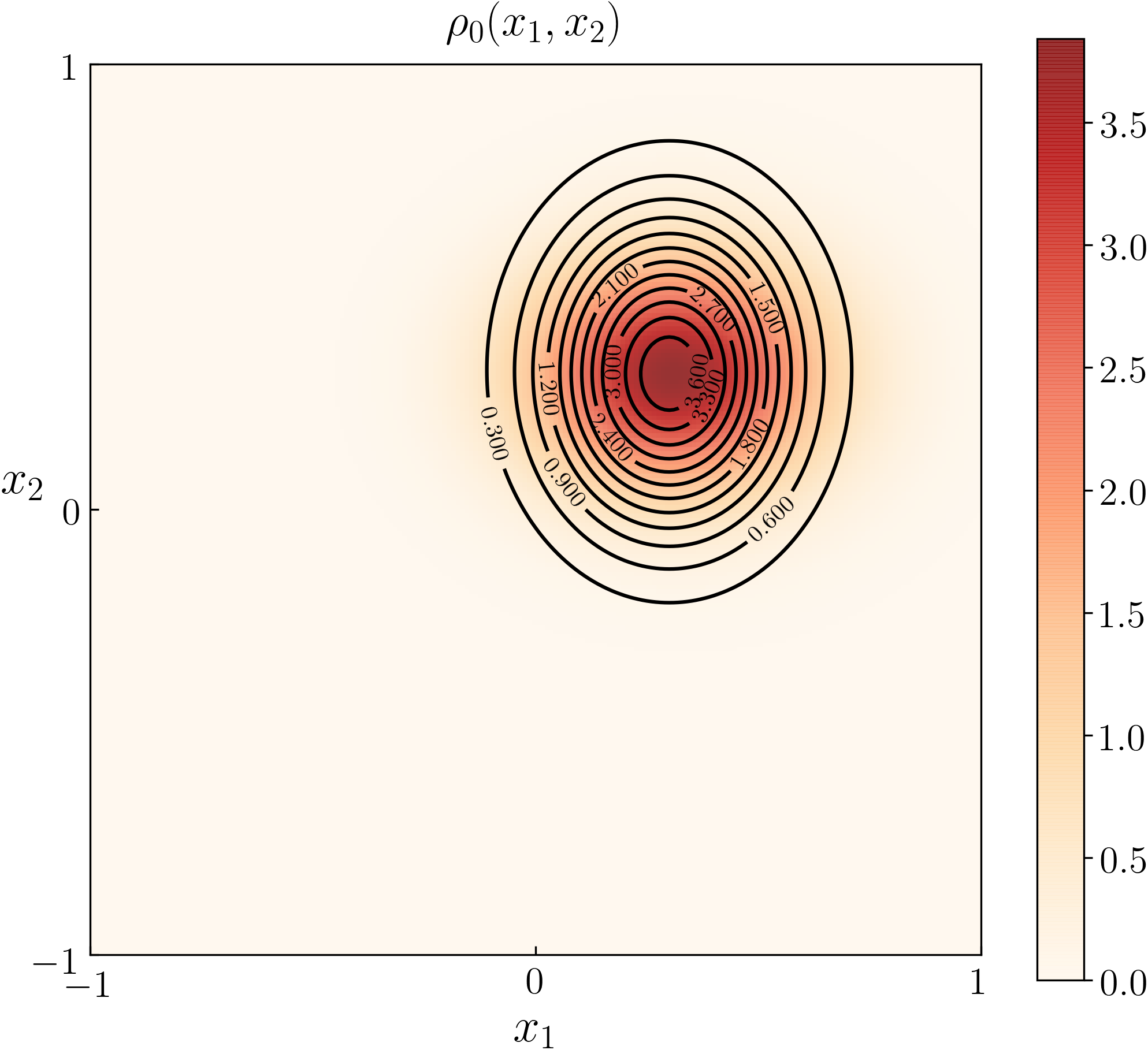}
        \caption{}
    \end{subfigure}%
\begin{subfigure}{0.5\textwidth}
    \centering
        \includegraphics[width=0.98\linewidth]{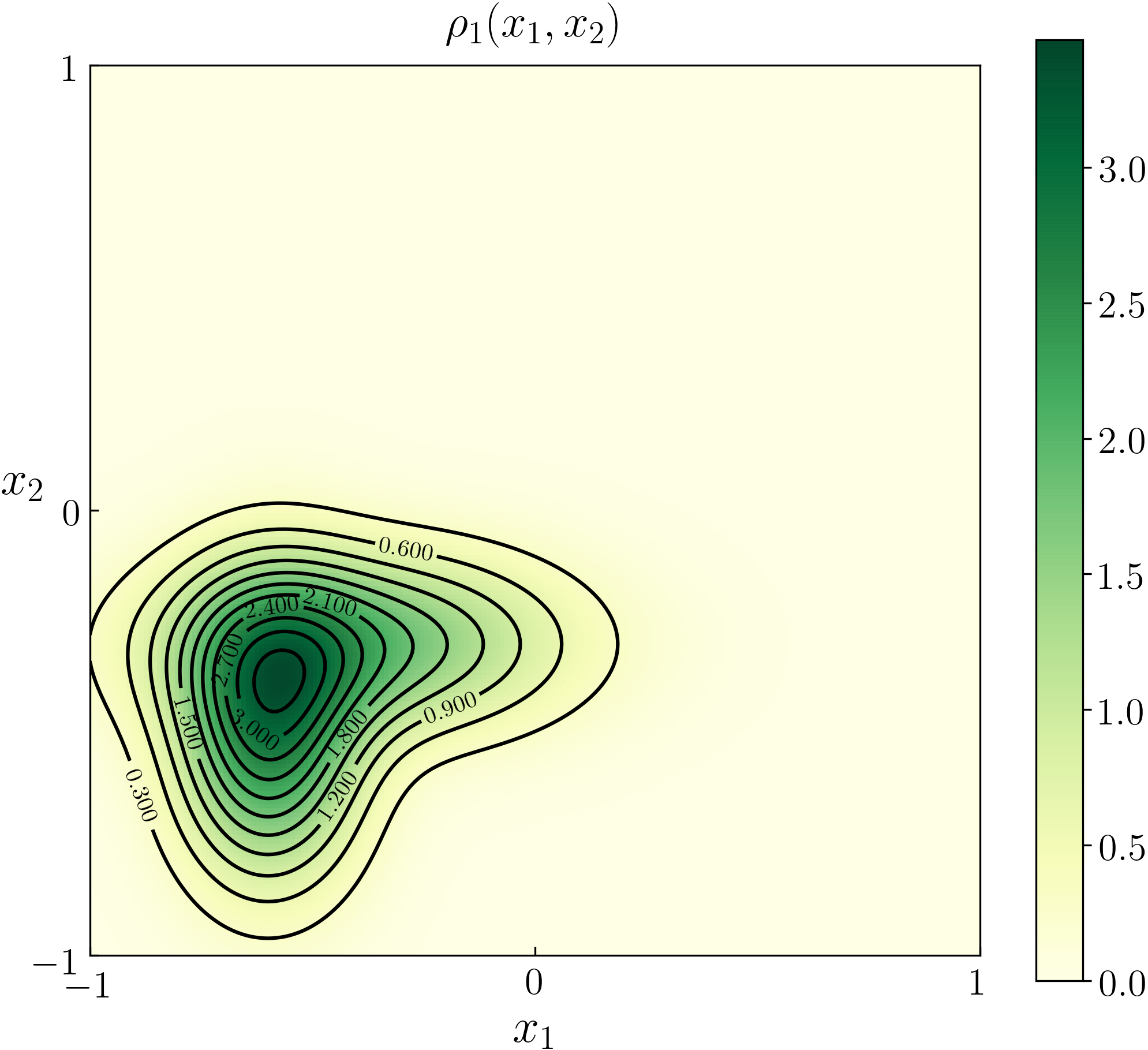}
        \caption{}
    \end{subfigure}%
\caption{\small{For Example 1, (a) the initial PDF $\rho_{0}$, and (b) the terminal PDF $\rho_{1}$. The corresponding colorbars show the (bivariate) joint PDF values.}}
\label{endpointPDFs2d}
\end{figure*}

\section{Numerical Examples}\label{SecNumExample}
In this Section, we provide two numerical examples of steering a prescribed joint state PDF $\rho_{0}(\bx)$ at $t=0$, to another prescribed $\rho_{1}(\bx)$ at $t=1$, subject to controlled dynamics $\dot{\bx}=\bm{f}(\bx)+\bm{g}(\bx)u$. In the following, Example 1 has two states and is easy to visualize while Example 2 has three states.

\begin{figure*}[t]
\centering
        \includegraphics[width=0.96\linewidth]{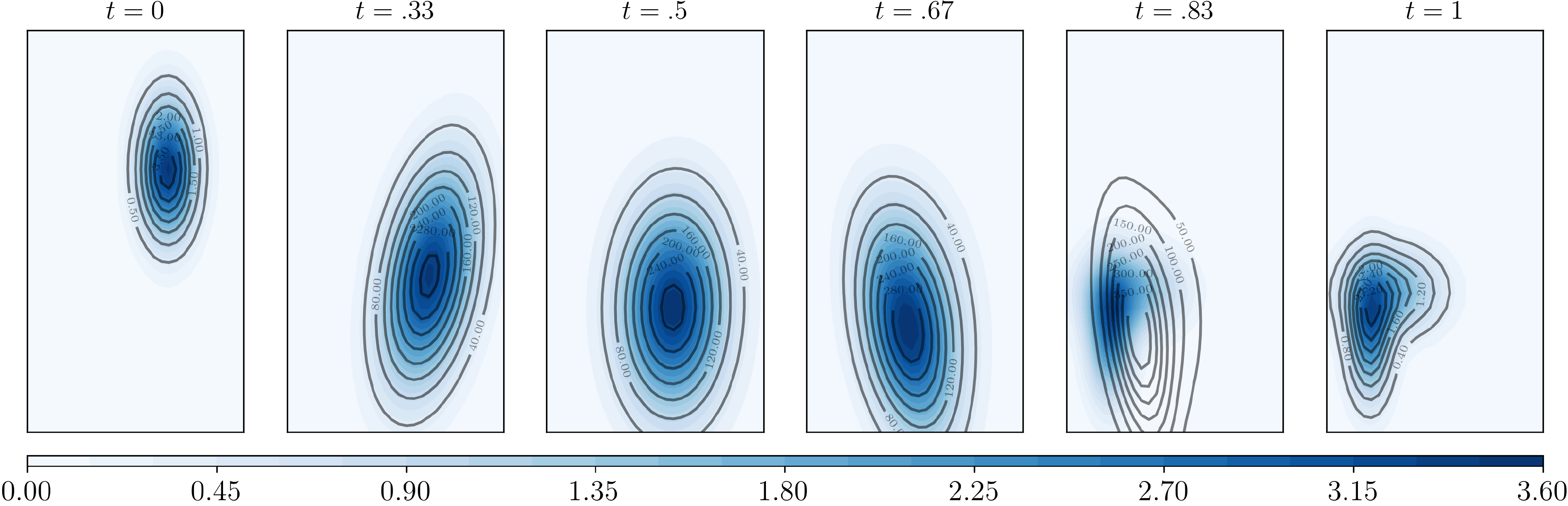}
\caption{\small{For Example 1, the $\epsilon$-regularized controlled transient joint state PDF $\rho_{\epsilon}(\bx,t)$ evolution is shown over $[-1,1]^{2}$, $t\in[0,1]$, computed using the framework described in Section \ref{SecSchrodinger}. The horizontal colorbar shows the joint PDF values. The transient evolution satisfies the endpoint data in Fig. \ref{endpointPDFs2d}.}}
\label{transPDFs2d}
\end{figure*}

\subsection{Example 1}\label{2state1controlExample}
We consider $\bm{x}\in\mathbb{R}^{2}$ with the single input controlled dynamics
\begin{align}
    \begin{pmatrix}
    \dot{x}_{1}\\
    \dot{x}_{2}	
    \end{pmatrix}
 =\! \underbrace{\begin{pmatrix}
     x_2 \\
-x_1+ \frac{1}{2}\left(1-x_{1}^2\right)x_{2}
    \end{pmatrix}}_{\bm{f}(\bx)} \!+\!
    \underbrace{\begin{pmatrix}
0  \\
     1
    \end{pmatrix}}_{\bm{g}(\bx)} u,
    \label{DiffFlatEx2d}
\end{align}
that satisfies the two conditions in Proposition \ref{ExistenceOflambda}, and we have
\begin{align}
	L_{\bm{g}} \lambda(\bx) = \dfrac{\partial\lambda}{\partial x_{2}} = 0, \quad L_{{\rm{ad}}_{\bm{f}}\bm{g}} \lambda(\bx) = \dfrac{\partial\lambda}{\partial x_{1}} + \dfrac{1}{2}\left(1-x_{1}^2\right)\dfrac{\partial\lambda}{\partial x_{2}} \neq 0,
	\end{align}	
for which $\lambda(\bx)=x_{1}$ is an admissible solution. Consequently, we obtain the trio of maps:
\begin{align}
	\bm{\tau}(\bx) = (x_{1}, x_{2})^{\top}, \quad \alpha(\bx) = -x_{1}+\frac{1}{2}(1-x_{1}^{2})x_{2}, \quad\beta(\bx)=1.
\end{align}	
In this case, $\mathcal{X}=\mathbb{R}^{2}$ since $\bm{\tau}$ is identity map. For $i=0,1$, we set
\begin{equation} \label{MargPDF1}
 \rho_{i}(\bx) := c_{1i} \mathcal{N}\left(\bm{\mu}_{1i},\bm{\Sigma}_{1i}\right) + c_{2i} \mathcal{N}\left(\bm{\mu}_{2i},\bm{\Sigma}_{2i}\right),
\end{equation}
with $(c_{10}, c_{20})=(0.19,0.81)$, $\bm{\mu}_{10}=(0.30,0.35)^{\top}$,
    $\bm{\mu}_{20}=(0.30,0.30)^{\top}$,
    $\bm{\Sigma}_{10} =\rm{diag}(0.05 ,0.067)$, 
    $\bm{\Sigma}_{20}=\rm{diag}(0.03,0.05)$, and $(c_{11}, c_{21})=(0.50,0.50)$, $\bm{\mu}_{11}=(-0.40,-0.30)^{\top}$,
    $\bm{\mu}_{21}=(-0.60,-0.50)^{\top}$,
    $\bm{\Sigma}_{11} =\rm{diag}(0.095 ,0.02)$, 
    $\bm{\Sigma}_{21}=\rm{diag}(0.02,0.05)$. The endpoint PDFs $\rho_{0},\rho_{1}$ thus generated, are shown in Fig. \ref{endpointPDFs2d}.

In Fig. \ref{transPDFs2d}, we show the $\epsilon$-regularized controlled joint state PDF $\rho_{\epsilon}(\bx,t)$ evolution for $t\in[0,1]$ and $\bx\in[-1,1]^{2}$ subject to (\ref{DiffFlatEx2d}) and endpoint data shown in Fig. \ref{endpointPDFs2d}. These were computed using the framework described in Section \ref{SecSchrodinger} with the regularization parameter $\epsilon = 10^{-3}$.  

\subsection{Example 2}\label{3state1controlExample}
We now consider $\bm{x}\in\mathbb{R}^{3}$ with the controlled dynamics 
\begin{align}
    \begin{pmatrix}
    \dot{x}_{1}\\
    \dot{x}_{2}\\
    \dot{x}_{3}	
    \end{pmatrix}
 =\! \underbrace{\begin{pmatrix}
     x_3 \\
     -x_2 \\
     -x_1+x_{2}-2x_2^2
    \end{pmatrix}}_{\bm{f}(\bx)} \!+\!
    \underbrace{\begin{pmatrix}
     -x_2 \\
     1  \\
     2x_2
    \end{pmatrix}}_{\bm{g}(\bx)} u,
    \label{DiffFlatEx}
\end{align} 
which satisfies the two conditions in Proposition \ref{ExistenceOflambda}, and hence is static state feedback linearizable. In particular, (\ref{lambdacond}) yields
\begin{subequations}
	\begin{align}
	L_{\bm{g}} \lambda(\bx) &= -x_2 \frac{\partial \lambda}{\partial x_1} + \frac{\partial \lambda}{\partial x_2} + 2x_2 \frac{\partial \lambda}{\partial x_3} = 0,\\
      L_{{\rm{ad}}_{\bm{f}}\bm{g}} \lambda(\bx) &= -x_2 \frac{\partial \lambda}{\partial x_1} + \frac{\partial \lambda}{\partial x_2} + (1-x_2) \frac{\partial \lambda}{\partial x_3} = 0,\\
      L_{{\rm{ad}}_{\bm{f}}^{2}\bm{g}} \lambda(\bx) &= \frac{\partial \lambda}{\partial x_1} + \frac{\partial \lambda}{\partial x_2} + (2x_{2}-1)\frac{\partial \lambda}{\partial x_3} \neq 0, 
\end{align}
\end{subequations}
for which $\lambda(\bx) = x_1 + x_2^{2}/2$ is an admissible solution in 
\begin{align}
\mathcal{X}:=\{\bx\in\mathbb{R}^{3}\mid x_{2}\neq -1\}.
\label{XsetExample}	
\end{align}
Furthermore, in this case,
\begin{equation}
    \begin{aligned}
    L_{\bbf}\lambda(\bx) &= x_3-x_2^2, \quad L_{\bbf}^2\lambda(\bx) = -x_1+x_2, \\
    L_{\bbf}^3\lambda(\bx) & = -x_3-x_2,    \quad L_{\bm{g}}L_{\bbf}^{2}\lambda(\bx) = 1+x_2,
    \end{aligned}
\end{equation}
using which in (\ref{alphabetatau}), results in the trio of maps
\begin{equation}
\begin{aligned}
 \bm{\tau}(\bx) &= \left(x_1+x_{2}^{2}/2,\;-x_{2}^{2}+x_{3},\;-x_{1}+x_{2}\right)^{\top}, \\ 
 \alpha(\bx) &= \frac{x_2+x_{3}}{1+x_2}, \quad \beta(\bx) = \frac{1}{1+x_2}.
\end{aligned}
\label{ExampleTuple}
\end{equation}
The feedback linearization of (\ref{DiffFlatEx}) via $(\bm{\tau},\alpha,\beta)$ as in (\ref{ExampleTuple}) is valid for all $(x_{1},x_{2},x_{3})\in\mathbb{R}^{3}$ such that $\det(\nabla_{\bx}\bm{\tau}) = -1-x_{2}\neq 0$, i.e., in $\mathcal{X}$ given by (\ref{XsetExample}).

The plane $x_{2} = -1$ splits $\mathcal{X}$ into two disjoint regions: $\mathcal{X}_L$ to the left of the plane, and $\mathcal{X}_R$ to the right of the plane. For $\bz=\bm{\tau}(\bx)$, the feedback linearized form is $\dot{\bz} = \bA \bz + \bb v$ with $\bA =\left[\bm{0}\vert \bm{e}_{1} \vert \bm{e}_{2}\right] $ and $\bb= \bm{e}_3$. Since the feedback linearization in original coordinates is valid in $\mathcal{X}=\mathcal{X}_L\cup\mathcal{X}_R$, in the numerical simulation, we choose the end point PDFs $\rho_{0}(\bx),\rho_{1}(\bx)$ so that the PDF evolution remains in $\mathcal{X}_R$. In particular, we fix $\rho_{0}$ as in (\ref{MargPDF1}) where $(c_{10}, c_{20})=(0.19,0.81)$, $\bm{\mu}_{10}=(0.30,0.35,0.50)^{\top}$,
    $\bm{\mu}_{20}=(0.30,0.30,0.50)^{\top}$,
    $\bm{\Sigma}_{10} =\rm{diag}(0.05 ,0.067 ,0.04)$, 
    $\bm{\Sigma}_{20}=\rm{diag}(0.03,0.05, 0.05)$.
Our goal is to steer the two component Gaussian mixture $\rho_{0}$ to another two component Gaussian mixture
\begin{equation}\label{MargPDF2}
  \rho_1(\bx) :=c_{11} \mathcal{N}\left(\bm{\mu}_{11},\bm{\Sigma}_{11}\right) + c_{21} \mathcal{N}\left(\bm{\mu}_{21},\bm{\Sigma}_{21}\right),
\end{equation}
with $(c_{11}, c_{21})=(0.39,0.61)$, $\bm{\mu}_{11}=(0.50,0.40,0.30)^{\top}$,
    $\bm{\mu}_{21}=(0.80,0.60,0.40)^{\top}$,
    $\bm{\Sigma}_{11}=\rm{diag}(0.095,0.02,0.04)$,
    $\bm{\Sigma}_{21}=\rm{diag}(0.02,0.05,0.04)$. 

Given $\bm{f},\bm{g}$ as in (\ref{DiffFlatEx}), and $\rho_{0},\rho_{1}$ as in (\ref{MargPDF1})-(\ref{MargPDF2}), instead of solving (\ref{FLSdensityprob}) or its equivalent (\ref{TransformedLOMT}), we seek a numerical solution for (\ref{TransOCFluidsrelax}), i.e., find a \emph{feasible} controller that steers prescribed $\rho_{0}$ to $\rho_{1}$ in unit time subject to given feedback linearizable dynamics. To this end, we numerically solve (\ref{RelaxedSchrodinger}), which is the dynamic stochastic regularization of (\ref{TransOCFluidsrelax}) for small $\epsilon$ (here, $\epsilon=0.01$). The resulting optimal regularized controlled PDF $\widetilde{\sigma}_{\epsilon}^{\rm{opt}}(\bz,t)$ in (\ref{SBPsolnPDF}) is then mapped back to original state space $\mathcal{X}$ as the pushforward 
\begin{align}
\bm{\tau}^{-1}\:\sharp\:\widetilde{\sigma}_{\epsilon}^{\rm{opt}} =: \rho_{\epsilon}(\bx,t).
\label{FinalPush}	
\end{align}

 Shown in Fig. \ref{ZcutTransients} are the joint PDFs $\rho_{\epsilon}(\bx,t)$ in $\mathcal{X}$ associated with the dynamics (\ref{DiffFlatEx}), wherein the PDFs $\widetilde{\sigma}_{\epsilon}^{\rm{opt}}(\bz,t)$ are computed by solving the associated Schr\"{o}dinger system mentioned in Section \ref{SecSchrodinger}. To recap, the steps for computing the $\epsilon$-regularized controlled PDF $\rho_{\epsilon}(\bx,t)$ are the following: 
 \begin{equation*}
\begin{tikzpicture}[baseline=-0.8ex]
    \matrix (m) [
            matrix of math nodes,
            row sep=5em,
            column sep=6.2em,
            text height=1.5ex, text depth=0.25ex
            ] {
        \left(\rho_{0},\rho_{1}\right) & \left(\sigma_{0},\sigma_{1}\right) &  \left(\hat{\sigma}_{0},\hat{\sigma}_{1}\right)   \\
        (\hat{h},h) & (\hat{h}_{0},h_{1}) & (\hat{h}_{0}^{{\rm{B}}},h_{1}^{{\rm{B}}}) \\
        \widetilde{\sigma}_{\epsilon}^{\rm{opt}}(\bz,t) & \rho_{\epsilon}(\bx,t). & \\
        };
    \path[->]
        (m-1-1) edge node[above] {(\ref{rho2sigma})} (m-1-2)
        (m-1-2) edge node[above]  {(\ref{OMTrelaxMarg})} (m-1-3)
        (m-1-3) edge node[left] {(\ref{IntegralEquations}) with $\kappa\mapsto\kappa^{{\rm{B}}}$} (m-2-3)
        (m-2-3) edge node[above] {(\ref{RecoverLinSchrodingerSystemSoln})} (m-2-2)
        (m-2-2) edge node[above] {(\ref{SBPsolnForwardKernel})-(\ref{SBPsolnBackwardKernel})} (m-2-1)
        (m-2-2) edge node[below] {use $\kappa$ from (\ref{kappaViaBrownianKernel})} (m-2-1)
        (m-2-1) edge node[left] {(\ref{SBPsolnPDF})} (m-3-1)
        (m-3-1) edge node[above] {(\ref{FinalPush})} (m-3-2);
\end{tikzpicture}   
\end{equation*}
\noindent The results in Fig. \ref{ZcutTransients} are depicted for six different time snapshots, with a uniform spatial discretization having 1000 grid points. A different view of the same plot is shown in Fig. \ref{YcutTransients}.

\begin{figure*}[t]
\centering
\includegraphics[width=.95\textwidth]{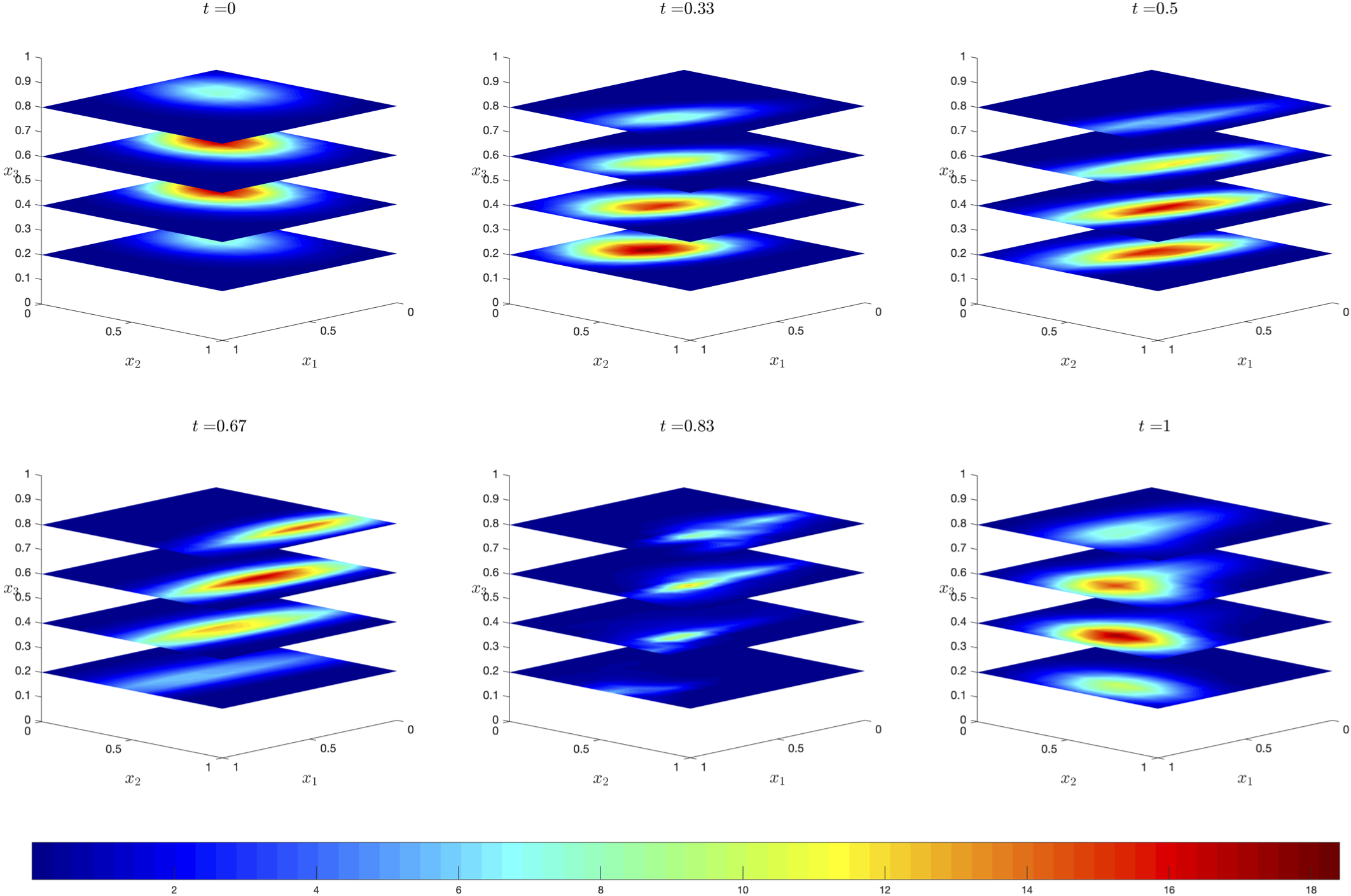}
\caption{\small{The colormap of the $\epsilon$-regularized controlled transient joint PDFs $\rho_{\epsilon}(\bx,t)$ for (\ref{DiffFlatEx}) with endpoint joint PDFs $\rho(\bx,0)=\rho_0(\bx)$ and $\rho(\bx,1)=\rho_1(\bx)$ given by (\ref{MargPDF1})-(\ref{MargPDF2}), resulting from the fixed point recursion of (\ref{IntegralEquations}) associated with (\ref{RelaxedSchrodinger}). Here, $\epsilon=0.01$. At any fixed $t\in[0,1]$, the joint PDF $\rho_{\epsilon}(\bx,t)$ is supported over three spatial dimensions in the state space $\mathcal{X}$. Different subplots correspond to different time snapshots. In each subplot, we take our planar slices at $x_3=0.2, 0.4, 0.6, 0.8$ to show the PDF evolution around the same. The color (\emph{red} = high, \emph{blue} = low) denotes the joint PDF value (see colormap).}}
\label{ZcutTransients}
\end{figure*}

\begin{figure*}[t]
\centering
\includegraphics[width=.95\textwidth]{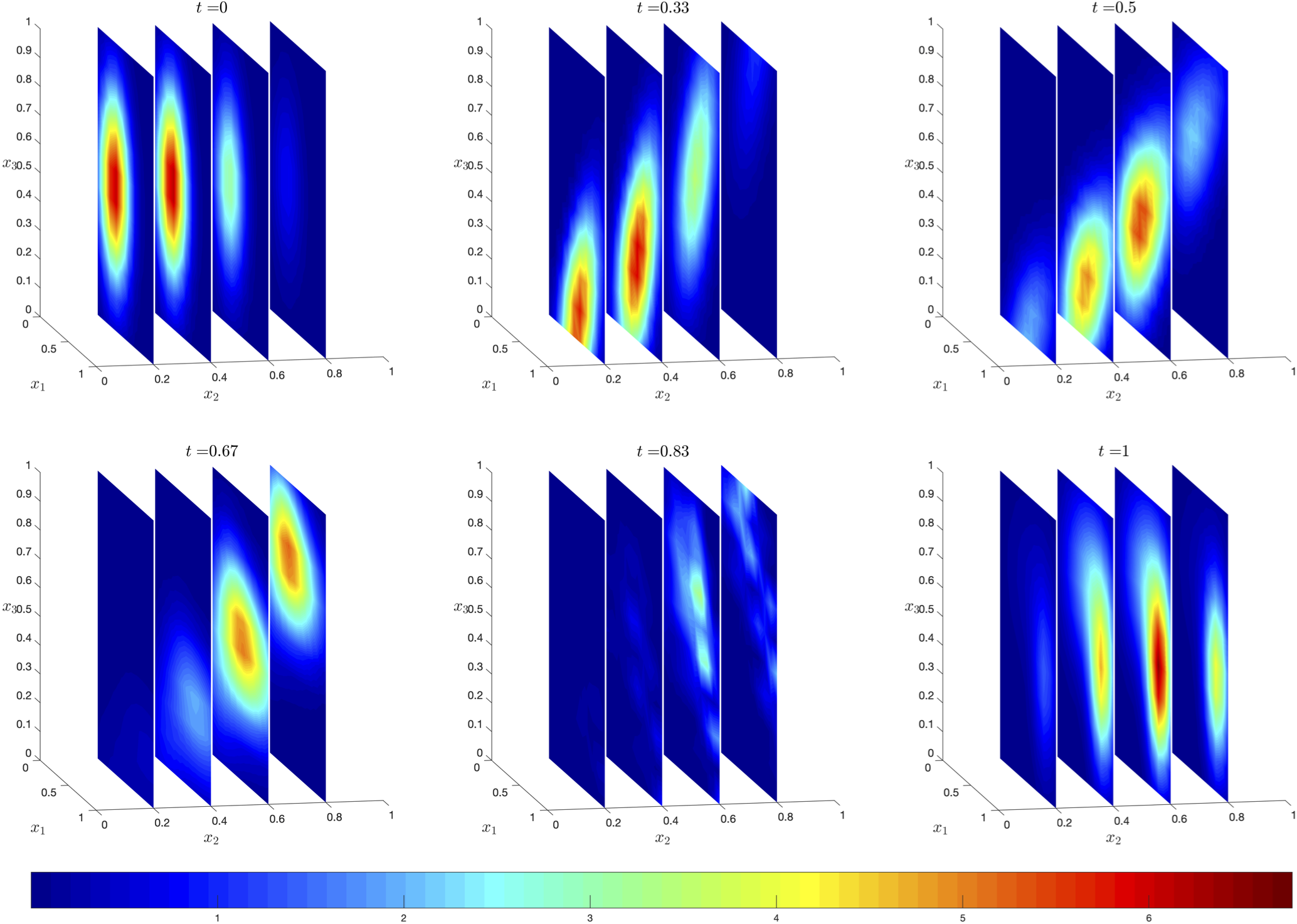}
\caption{\small{As in Fig. \ref{ZcutTransients}, shown above are colormaps of the $\epsilon$-regularized controlled transient joint PDFs $\rho_{\epsilon}(\bx,t)$ for (\ref{DiffFlatEx}) with endpoint joint PDFs $\rho(\bx,0)=\rho_0(\bx)$ and $\rho(\bx,1)=\rho_1(\bx)$ given by (\ref{MargPDF1})-(\ref{MargPDF2}), with same parameters as before. Here, we take four planar slices at $x_2=0.2, 0.4, 0.6, 0.8$ to show the PDF evolution around the same. The color (\emph{red} = high, \emph{blue} = low) denotes the joint PDF value (see colormap).}}
\label{YcutTransients}
\end{figure*}


\section{Conclusions} 
The problem of steering the state of a nonlinear control system from a prescribed joint PDF $\rho_{0}$ to another $\rho_{1}$, over a finite time horizon via feedback control is currently open in the literature. This atypical stochastic control problem can be seen as a measure-valued two-point boundary value problem subject to the controlled nonlinear dynamics. Motivated by the observation that many envisaged applications of this problem (e.g., probabilistic path planning for ground and aerial robots, swarm guidance) involve feedback linearizable systems, this paper presents the theory and computational algorithms for finite horizon density control for single input systems with full state feedback. We present theoretical characterization of the minimum energy feedback controller that reshapes the joint state PDF $\rho_{0}$ at time $t=0$ to $\rho_{1}$ at time $t=1$. Harnessing the recently established connections between stochastic control, the theory of optimal transport and the Schr\"{o}dinger bridge, we provide numerical algorithms to generate feasible controllers.

Several extensions are possible, e.g., generalizations for the multi-input static state feedback linearizable systems, and for the dynamic state feedback linearizable systems. The authors' recent work \cite{caluya2019finite} pursues the former. Another natural direction is to incorporate the input and the state (i.e., path inequality) constraints in problem (\ref{FLSprobMinEnergy}). We note that input constraints for finite horizon covariance control in linear systems were accounted recently in \cite{bakolas2018finite,okamoto2019input}. The recent papers \cite{okamoto2018optimal,ridderhof2019nonlinear,ridderhof2020chance} have addressed probabilistic state constraints for optimal covariance control. Another possible extension is to account more general quadratic cost in problem (\ref{FLSprobMinEnergy}), including the state penalty. To the best of the authors' knowledge, so far this has only been addressed for the controlled linear dynamics \cite{chen2018optimal}. Also, algorithms for the direct numerical solution of the minimum energy controller would be of interest. These will be pursued in the follow-up works of the authors. 


\appendix

\subsection{Proof of Theorem \ref{ThmHJBopt}}\label{appendix:deriveHJB}
\begin{proof}[\unskip\nopunct]
We rewrite the Lagrangian (\ref{Lagrangian}) as
\begin{align}
&\mathscr{L}(\sigma,\psi,v) = \int_{\mathcal{Z}} \int_{0}^{1}\frac{1}{2} \mathcal{L}(\bz,v) \:\sigma(\bz,t)\differential t\differential\bz  \nonumber\\
&+  \underbrace{\int_{\mathcal{Z}}\left(\int_{0}^{1}\psi(\bz,t)\dfrac{\partial \sigma}{\partial t}\differential t\right)\differential\bz}_{\text{term 1}} \nonumber\\
&+ \underbrace{\int_{0}^{1} \left(\int_{\mathcal{Z}}	\psi(\bz,t)\nabla_{\bz}\cdot\left(\left(\bm{A}\bz + \bm{b}v\right)\sigma(\bz,t)\right) \differential\bz\right)\differential t}_{\text{term 2}},						
\label{LagrangianExpanded}
\end{align}
and perform integration by parts in variable $t$ for term 1, and in variable $\bz$ for term 2, as indicated in (\ref{LagrangianExpanded}). Noting that term 1 yields
\[\underbrace{\int_{\mathcal{Z}} \left( \psi(\bm{z},1)\sigma_1(\bm{z})-  \psi(\bm{z},0)\sigma_0(\bm{z} \right) \differential t \ \differential\bm{z} }_{\text{constant w.r.t to $(\sigma,v)$ }}   -\int_{\mathcal{Z}} \int_{0}^{1}\frac{\partial \psi}{\partial t} \sigma \differential t \ \differential\bm{z},\]
we get that up to an additive constant, $\mathscr{L}$ equals
\begin{align}
\!\displaystyle\int_{\mathcal{Z}}\!\int_{0}^{1}\!\!\bigg\{\!\frac{1}{2}\mathcal{L}(\bz,v) - \dfrac{\partial \psi}{\partial t} - \langle\nabla_{\bz}\psi,\left(\bm{A}\bm{z}+\bb v\right)\rangle\!\bigg\}\sigma(\bz,t)\differential\bz\:\differential t.
\label{AfterIBP}	
\end{align}
Point-wise minimization of (\ref{AfterIBP}) w.r.t. $v$ gives the optimal control
\begin{align}
v^{\rm{opt}}(\bz,t) = \frac{\bb^{\top}\nabla_{\bz}\psi}{\beta_{\bm{\tau}}^{2}(\bz)} - \frac{\alpha_{\bm{\tau}}(\bm{z})}{\beta_{\bm{\tau}}(\bm{z})},
\label{OptControlv}	
\end{align}
which simplifies to (\ref{voptOptimal}) since $\bb=\bm{e}_{n}$. The pointwise minimization of the augmented Lagrangian can be justified via the Lagrange multiplier theorem in Banach spaces (see \cite[Ch. 4.14, Proposition 1]{zeidler1995applied}). Specifically, we define the function spaces 
\begin{align*}
\overline{\mathcal{S}}_{01} &:= \mathcal{S}_{01} \cap L^{2}(H^{1}(\mathcal{Z});[0,1]) \cap \dot{H}^{1}\left(\left(H^{1}(\mathcal{Z})\right)^{*};[0,1]\right),\\
X &:= \overline{\mathcal{S}}_{01} \times  L^{2}(\mathcal{Z} \times [0,1]), \quad Y:=  L^{2}(H^{-1}(\mathcal{Z});[0,1]),
\end{align*}
where $[0,1]$ denotes the time interval, $\mathcal{S}_{01}$ is defined as in (20), and $L^{2}\left(\cdot\right)$ denotes the space of square integrable functions. The notation $L^{2}(H^{1}(\mathcal{Z});[0,1])$ stands for the Sobolev space of functions having first order weak derivatives w.r.t. $\bm{z}\in\mathcal{Z}$, and finite $L^{2}$ norms w.r.t. $t\in[0,1]$. Furthermore, $\dot{H}^{1}\left(\left(H^{1}(\mathcal{Z})\right)^{*};[0,1]\right):= \{f(\cdot,t)\in L^{2}\left([0,1]\right) \mid \frac{\partial f}{\partial t} \in L^{2}\left([0,1]\right),f\in\left(H^{1}(\mathcal{Z})\right)^{*}\}$, wherein $\left(H^{1}(\mathcal{Z})\right)^{*}$ denotes the dual space of the Sobolev space $H^{1}(\mathcal{Z})$. We denote the dual space of $\overline{\mathcal{S}}_{01}$ as $\overline{\mathcal{S}}_{01}^{*}$. In the definition of $Y$, the notation $H^{-1}\left(\mathcal{Z}\right)$ stands for the space of all linear functionals on $H_{0}^{1}\left(\mathcal{Z}\right):=\{f\in H^{1}\left(\mathcal{Z}\right), \; \text{and vanishes on}\; \partial\mathcal{Z}\}$. Then, in (14), the objective functional $F: X \mapsto \mathbb{R}$, and is given by 
\begin{align*}
F(\sigma,v) := \int_{\mathcal{Z}} \int_{0}^{1} \frac{1} {2}   \mathcal{L}(\bm{z},v) \sigma(\bm{z},t) \: \differential t \: \differential \bm{z}. 
\end{align*}
In (14), the constraint is a mapping $G: X \mapsto Y$ defined by 
\begin{align*}
&G(\sigma,v)(\psi) :=  \int_{\mathcal{Z}}\psi(\bm{z},1)\sigma(\bm{z},1) \differential \bm{z}- \int_{\mathcal{Z}} \psi(\bm{z},0)\sigma(\bm{z},0) \differential \bm{z} \nonumber \\
                        &-\int_{\mathcal{Z}} \int_{0}^{1}  \frac{\partial \psi }{\partial t} \sigma(\bm{z},t)  \differential \bm{z} \: \differential t + \int_{\mathcal{Z}} \int_{0}^{1}  \langle \nabla \psi, \bm{A} \bm{z} + \bm{b}v \rangle  \sigma(\bm{z},t) \differential \bm{z} \: \differential t. 
\end{align*}
If we can show that $G^{\prime}_{\sigma}(\sigma,v)$ and $G^{\prime}_{v}(\sigma,v)$ (where $^{\prime}$ denotes derivative w.r.t. the subscripted variable) are surjective, then by \cite[Ch. 4.14, Proposition 1]{zeidler1995applied}, there exists\footnote{Our development here ensures the spatial Sobolev regularity of $\psi$ in that it guarantees the existence of $\psi\in L^{2}\left(H_{0}^{1}(\mathcal{Z});[0,1]\right)$. We expect that the same treatment can be extended for the augmented vector $\left(t, \bm{z}\right)$ to guarantee $\psi\in L^{2}\left(H_{0}^{1}\left([0,1]\times\mathcal{Z}\right)\right)$. Pursuing this is beyond the scope of this paper.} $\psi\in Y^{*} := L^{2}(H_{0}^{1}(\mathcal{Z});[0,1])$ that allows the pointwise minimization of (75) resulting in (76).

To show the surjectivity of the aforesaid maps, we begin by calculating the functional derivatives:
\begin{subequations}
\begin{align}
&G^{\prime}_{\sigma}(\sigma,v)\left(\widetilde{\sigma},\psi\right) = -\int_{\mathcal{Z}} \int_{0}^{1}  \frac{\partial \psi }{\partial t} \widetilde{\sigma}(\bm{z},t)  \differential \bm{z} \: \differential t \nonumber\\
&+ \int_{\mathcal{Z}} \int_{0}^{1}  \langle \nabla \psi, \bm{A} \bm{z} + \bm{b}v \rangle  \widetilde{\sigma}(\bm{z},t) \differential \bm{z} \: \differential t,\;\text{for}\;\left(\widetilde{\sigma},\psi\right)\in \overline{\mathcal{S}}_{01}^{*}\times Y^{*},\label{Gprimesigma}\\
	&G^{\prime}_{v}(\sigma,v)\left(\widetilde{v},\psi\right) = \int_{\mathcal{Z}} \int_{0}^{1}  \langle \nabla \psi, \bm{b}\widetilde{v} \rangle \sigma(\bm{z},t) \differential \bm{z} \: \differential t,\nonumber\\
	&\qquad\qquad\qquad\quad\text{for}\;\left(\widetilde{v},\psi\right)\in L^{2}\left(L^{\infty}\left(\mathcal{Z}\right);[0,1]\right)\times Y^{*}, \label{Gprimev}
\end{align}
\end{subequations}
and then proceed similar to \cite[p. 112--114]{albi2017mean}. Specifically, for $(\sigma^{\rm{opt}},v^{\rm{opt}})$ solving (\ref{TransformedLOMT}), to show $G^{\prime}_{\sigma}(\sigma^{\rm{opt}},v^{\rm{opt}}) : \overline{\mathcal{S}}_{01}^{*} \mapsto Y$ is surjective, we need to show that for any $\eta\in Y$, there exists $\widetilde{\sigma}\in \overline{\mathcal{S}}_{01}^{*}$ such that $G^{\prime}_{\sigma}(\sigma,v)\left(\widetilde{\sigma},\psi\right) = \eta(\psi)$, for all $\psi\in Y^{*}$. This can be proved by showing the existence of solution for a linear parabolic PDE Cauchy problem as in \cite[p. 112]{albi2017mean} via a priori estimates. That $G^{\prime}_{v}(\sigma^{\rm{opt}},v^{\rm{opt}}):L^{2}\left(L^{\infty}\left(\mathcal{Z}\right);[0,1]\right)\mapsto Y$ is surjective, can be shown by a weak formulation of the Poisson equation followed by invoking the Riesz representation theorem as in \cite[p. 113]{albi2017mean}. We eschew the details here.

Substituting (\ref{OptControlv}) back in the expression within curly braces in (\ref{AfterIBP}) and then equating to zero yields the dynamic programming equation
\begin{align}
\inf_{v\in\mathcal{V}}\bigg\{\!\frac{1}{2}\mathcal{L}(\bz,v) - \dfrac{\partial \psi}{\partial t} - \langle\nabla_{\bz}\psi,\left(\bm{A}\bm{z}+\bb v\right)\rangle\!\bigg\} = 0,
\label{DPeqn}	
\end{align}
wherein using the arg inf (\ref{OptControlv}) followed by algebraic simplification using $\bm{A},\bb$ in (\ref{feedbacklinearizedform}), results in the HJB PDE (\ref{HJBoptimal}).
\end{proof}


\subsection{Rockafellar-Wolenski Envelope Representation Formula}\label{appendix:EnvelopeFormula}
\noindent We summarize a result from \cite{rockafellar2000convexity} useful in our context.
\begin{theorem}(see \cite[p. 1357, Theorem 2.6]{rockafellar2000convexity})\label{RWthm}
Consider the HJB PDE (\ref{HJBform}) with state-dependent Hamiltonian $\mathcal{H}(\bz,\bm{\zeta})$, and associated initial condition $\psi(\bz,0):=\psi_{0}(\bz)$. Suppose that $\mathcal{H}$ is finite, concave in $\bz$, convex in $\bm{\zeta}$, and that
\begin{enumerate}
\item[(H1)]	there exist constants $a,b$, and finite convex function $\varphi$ such that
\[\mathcal{H}(\bz,\bm{\zeta}) \leq \varphi(\bm{\zeta}) + \left(a\parallel\bm{\zeta}\parallel_{2} + b\right)\parallel\bz\parallel_{2},\quad\text{for all}\;\bz,\bm{\zeta},\]
\item[(H2)] there exist constants $c,d$, and finite convex function $\vartheta$ such that
\[\mathcal{H}(\bz,\bm{\zeta}) \geq -\vartheta(\bm{\zeta}) - \left(c\parallel\bm{\zeta}\parallel_{2} + d\right)\parallel\bz\parallel_{2},\quad\text{for all}\;\bz,\bm{\zeta}.\]
\end{enumerate}
Let 
\begin{align}
	\ell(\bz,\bm{w}) := \!\underset{\bm{\zeta}}{\sup}\{\langle\bm{w},\bm{\zeta}\rangle - \mathcal{H}(\bz,\bm{\zeta})\},
	\label{conjugate2ndargument}
\end{align}
and define the ``dualizing kernel" $K:[0,\infty)\times\mathbb{R}^{n}\times\mathbb{R}^{n} \mapsto \mathbb{R}\cup\{+\infty\}$ as
\[K(t,\bz,\bm{r}) := \underset{\bm{q}(t)}{\inf}\bigg\{\langle\bm{q}(0),\bm{r}\rangle + \!\int_{0}^{t}\!\!\ell\left(\bm{q}(t),\dot{\bm{q}}(t)\right)\differential t \mid \bm{q}(t)=\bz\bigg\},\]
and $K(0,\bz,\bm{r}):= \langle\bz,\bm{r}\rangle$. For arbitrary $\psi_{0}:\mathbb{R}^{n}\mapsto\mathbb{R}\cup\{+\infty\}$, the function $\psi(\bz,t)$ admits the following upper envelope representation:
\begin{align}
\psi(\bz,t) = \underset{\overline{\bz}}{\inf}\:\underset{\bm{r}}{\sup}\bigg\{\psi_{0}(\overline{\bz}) - \langle \overline{\bz}, \bm{r}\rangle + K\left(t,\bz,\bm{r}\right)\bigg\}.
\label{GeneralUpperEnvelope}	
\end{align}
Furthermore, if $\psi_{0}(\bz)$ is convex, proper and lower semicontinuous, then the function $\psi(\bz,t)$ admits the following lower envelope representation:
\begin{subequations}
\begin{align}
\psi(\bz,t) &= \underset{\bm{r}}{\sup}\:\underset{\overline{\bz}}{\inf}\bigg\{\psi_{0}(\overline{\bz}) - \langle \overline{\bz}, \bm{r}\rangle + K\left(t,\bz,\bm{r}\right)\bigg\}
\\
&= \underset{\bm{r}}{\sup}\bigg\{-\psi_{0}^{*}\left(\bm{r}\right) + K\left(t,\bz,\bm{r}\right)\bigg\}.\label{LowerEnvelopeConjugateform}
\end{align}
\label{GeneralLowerEnvelope}		
\end{subequations}
\end{theorem}
Notice that for $t=0$, formula (\ref{LowerEnvelopeConjugateform}) reduces to the bi-conjugate identity $\psi_{0}(\cdot)=\psi_{0}^{**}(\cdot)$ which is indeed valid when $\psi_{0}$ is convex and lower semicontinuous.


\subsection{Proof of Theorem \ref{ThmHJBoptHopfLax}}\label{appendix:DeriveEnvelope}
\begin{proof}[\unskip\nopunct]
	Using (\ref{alphabetatau}) and (\ref{defa}), we rewrite the Hamiltonian (\ref{HJBHamiltonian}) as
\begin{align}
\mathcal{H}\left(\bz,\bm{\zeta}\right) = \langle \bm{a}(\bz), \bm{\zeta} \rangle + \frac{1}{2}\left(L_{\bm{g}}L_{\bm{f}}^{n-1}\lambda(\bm{\tau}^{-1}(\bz))\right)^{2}\bm{\zeta}^{\top}\bb\bb^{\top}\bm{\zeta}.
\label{HamiltonianRewrite}	
\end{align}
Clearly, (\ref{HamiltonianRewrite}) is convex in $\bm{\zeta}$, and per assumption, is concave in $\bz$. 

To see that the conditions (H1)-(H2) in Theorem \ref{RWthm} hold for (\ref{HamiltonianRewrite}), notice that $v\in\mathcal{V}$ being finite energy (from Section \ref{SubsecReformulation}), the functions $\gamma,\delta$ in Remark \ref{u2vandv2u} must be bounded. Consequently, there exist positive scalars $c_{1},c_{2}$ such that $\lVert L^{n}_{\bm{f}} \lambda(\tau^{-1}(\bz))\rVert_{2} \leq c_{1}$, $\lVert \frac{1}{2}(L_gL_f^{n-1}\lambda(\tau^{-1}(\bz)))^2\rVert_{2} \leq c_{2}$, and thus, we get 
\begin{align}
\lvert\mathcal{H}\left(\bz,\bm{\zeta}\right)\rvert \leq \lVert\bA\rVert_{2} \lVert\bz\rVert_{2}\lVert\bm{\zeta}\rVert_{2} + c_{1} \lVert\bb\rVert_{2} \lVert\bm{\zeta} \rVert_{2} + c_{2} \lVert \bb \rVert_{2}^2  \lVert \bm{\zeta} \rVert_{2}^2.  
\label{boundH1}	
\end{align}
Recalling that $\lVert\bA\rVert_{2}=\lVert\bb\rVert_{2}=1$, from (\ref{boundH1}), we have that (H1) holds with $a=1,b=0$, and $\varphi(\bm{\zeta})=c_{1} \lVert\bm{\zeta} \rVert_{2} + c_{2} \lVert \bm{\zeta} \rVert_{2}^2$, for some $c_{1},c_{2}>0$. Unpacking the absolute value in (\ref{boundH1}) yields the desired inequality for condition (H2). 

Finally, to apply Theorem \ref{RWthm}, notice that in our case, the convex conjugate (\ref{conjugate2ndargument}) w.r.t. the argument $\bm{\zeta}$ is (\ref{TrueLagrangian}); see e.g., \cite[p. 108]{rockafellar1970convex}. Hence the statement.
\end{proof}


\subsection{Explicit Formula for $\bm{M}_{10},\bm{\Phi}_{10}$ in (\ref{OMTrelaxMarg})}\label{appendix:STMGramianFormula}
\noindent Given the linear time invariant system (\ref{feedbacklinearizedform}), for $0\leq s<t\leq 1$, its state transition matrix is 
\begin{align}
\bm{\Phi}_{ts}:=\bm{\Phi}(t,s) = \exp(\bA(t-s)),
\label{GeneralSTM}
\end{align} 
where $\exp(\cdot)$ denotes the matrix exponential. Since $\bA$ has nilpotency order $n$ (i.e., $\bA^{n} =\bm{0}$), hence the matrix $\bm{\Phi}_{10}=\exp(\bA)$ is upper-triangular with components 
\begin{equation}
  \exp(\bA)_{i,j} =  \begin{cases}
  \dfrac{1}{(j-i)!} & \text{for } i<j, \\
    1 & \text{for }   i=j, \\
    0 & \text{for }   i>j,
  \end{cases}
  \label{STMexplicit}
  \end{equation}
  where $i,j=1,\hdots,n$. Likewise, $\bm{\Phi}_{10}^{-1} = \exp(-\bA)$ appearing in (\ref{OMTrelaxMarg0}) is upper-triangular with elements which are negated of the same for $\bm{\Phi}_{10}$.
%
Recalling that the controllability Gramian is given by 
\begin{equation*}
\bM_{ts}:=\bM(t,s) = \!\int_{s}^{t}\!\!\exp(\bA(t-\tau)) \bb \bb^{\top}\exp(\bA(t-\tau))^{\top} \differential\tau,
\end{equation*}
direct calculation yields
\begin{align}
(\bM_{10})_{i,j} = \frac{1}{(n-i)!(n-j)!(2n-i-j+1)},
\label{ControllabGram10Explicit}	
\end{align}
for $i,j=1,\hdots,n$. In general, for $0\leq s<t\leq 1$, we have 
\begin{align}
(\bM_{ts})_{i,j} = \dfrac{(t-s)^{2n-i-j+1}}{(n-i)!(n-j)!(2n-i-j+1)},
\label{ControllabGram10Explicit}	
\end{align}
for $i,j=1,\hdots,n$.



\subsection{Existence and Uniqueness of Solution for (\ref{psiviaz0t})-(\ref{z0viaz})}\label{appendix:MOCinvfnthm}
Let $\bm{N}(t):=\exp(-t\bA)\bM(t,0)\exp(-t\bA^{\top}) \succ 0$, and note from  (\ref{z0viaz}) that the Jacobian
\begin{align}
\nabla_{\bz_{0}}\bz = \exp(t\bA)\bigg\{\bm{I} + \bm{N}(t)\bm{R}^{\top}\left[{\rm{Hess}}\left(\phi\left(\bm{R}\bz_{0}\right)\right) - \bm{I}\right]\bm{R}\bigg\}.
\label{Jaczz0}	
\end{align}
Since the spectrum of $\bm{N}(t)\bm{R}^{\top}\left[{\rm{Hess}}\left(\phi\left(\bm{R}\bz_{0}\right)\right) - \bm{I}\right]\bm{R}$ is same as that of $\bm{S}^{1/2}(t)\left[{\rm{Hess}}\left(\phi\left(\bm{R}\bz_{0}\right)\right) - \bm{I}\right]\bm{S}^{1/2}(t)$, where $\bm{S}(t) := \bm{R}\bm{N}(t)\bm{R}^{\top}\succ 0$, therefore from (\ref{STMexplicit}) and (\ref{Jaczz0}), we have that 
\[\det\left(\nabla_{\bz_{0}}\bz\right) = \prod_{i=1}^{n}\!\left(\!1+\lambda_{i}\!\left(\bm{S}^{1/2}(t)\left[{\rm{Hess}}\left(\phi\left(\bm{R}\bz_{0}\right)\right) - \bm{I}\right]\bm{S}^{1/2}(t)\right)\!\right)\]
wherein $\lambda_{i}(\cdot)$ denotes the $i$-th eigenvalue. Assuming that the endpoint PDFs $\hat{\sigma}_{0},\hat{\sigma}_{1}$ are bounded away from zero so that (\ref{MongeAmpere}) is well-defined, we have ${\rm{Hess}}(\phi)\succ\bm{0}$, and thus the $\lambda_{i}(\cdot)$ terms above are $> -1$ for all $i=1,\hdots,n$. Hence $\nabla_{\bz_{0}}\bz$ is nonsingular for all $\bz_{0}\in\mathcal{Z}$. By inverse function theorem, this confirms the existence of solution for (\ref{z0viaz}) for all $\bz\in\mathcal{Z}$, and hence the same for (\ref{psiviaz0t}). 

To see uniqueness, let $\bm{r}_{0} := \bm{R}\bz_{0}$, and rewrite (\ref{z0viaz}) as
\begin{align}
\left(\bm{I} - \bm{S}(t)\right)\bm{r}_{0} + \bm{S}(t)\hat{T}^{{\rm{opt}}}\left(\bm{r}_{0}\right) = \bm{R}\exp(-t\bA)\bz.
\label{r0eqn}	
\end{align}
Consider if possible that (\ref{z0viaz}) admits two solutions $\widetilde{\bz}_{0}\neq\overline{\bm{z}}_{0}$, or equivalently (since $\bm{R}$ is nonsingular) that (\ref{r0eqn}) admits two solutions $\widetilde{\bm{r}}_{0}\neq\overline{\bm{r}}_{0}$, for fixed $\bz$ and $t\in[0,1]$. Notice that $\bm{0}=\bm{S}(0) \preceq \bm{S}(t) \preceq \bm{I}=\bm{S}(1)$ for all $t\in[0,1]$. Substituting the two candidate solutions $\widetilde{\bm{r}}_{0},\overline{\bm{r}}_{0}$ for $\bm{r}_{0}$ in (\ref{r0eqn}), and subtracting the resulting two equations, we get 
\begin{align}
\left(\bm{I} - \bm{S}(t)\right)\left(\widetilde{\bm{r}}_{0}-\overline{\bm{r}}_{0}\right) + \bm{S}(t)\left(\hat{T}^{{\rm{opt}}}\left(\widetilde{\bm{r}}_{0}\right) - \hat{T}^{{\rm{opt}}}\left(\overline{\bm{r}}_{0}\right)\right) = \bm{0},
\label{Contradiction}	
\end{align}
which gives $\widetilde{\bm{r}}_{0}=\overline{\bm{r}}_{0}$ at both $t=0,1$, resulting in a contradiction. The contradiction at $t\in(0,1)$ can also be obtained by recalling the convexity of $\phi$, where $\hat{T}^{{\rm{opt}}}(\cdot) = \nabla\phi(\cdot)$. 


%
\bibliographystyle{IEEEtran}
\bibliography{references.bib}


%
%

\end{document}